\documentclass[12pt,A4paper]{article}
\usepackage[a4paper]{geometry} 
\usepackage{microtype}
\usepackage[T1]{fontenc}
\usepackage[utf8]{inputenc}
\usepackage{lmodern}
\usepackage[hidelinks=true,hypertexnames=false]{hyperref}
\usepackage[table,xcdraw]{xcolor}

\usepackage{xcolor}
\usepackage{graphicx}
\usepackage{color}
\usepackage{subfig}
\usepackage{float}

\usepackage{tikz}	
\usetikzlibrary{shapes,arrows,fit}
\usetikzlibrary{positioning}

\usepackage{bm}
\usepackage{amsmath}
\usepackage{amssymb}
\usepackage{booktabs}
\usepackage{color}
\usepackage{comment}
\usepackage[firstinits=true]{biblatex}
\usepackage{mathtools}
\usepackage{orcidlink}
\newtheorem{theorem}{Theorem}[section]
\newtheorem{lemma}[theorem]{Lemma}
\newtheorem{proposition}[theorem]{Proposition}
\newtheorem{corollary}[theorem]{Corollary}
\newtheorem{remark}[theorem]{Remark}
\newtheorem{definition}[theorem]{Definition}

\newenvironment{proof}{\paragraph{Proof.}}{\hfill$\blacksquare$}

\def\lD{{\cal D}}

\def\lM{{\cal M}}

\def\lR{{\cal R}} 
\def\lS{{\cal S}} 
\def\lT{{\cal T}}

\def\mA{{\bf A}} 
\def\mB{{\bf B}}

\def\mj{{\bf j}}

\def\mq{{\bf q}}

\def\mv{{\bf v}}

\def\vg{{\sf g}}

\mathchardef\kvtpr="2D02
\mathchardef\kvadratic="0D03
\mathchardef\punikv="0D04
\mathchardef\rnsmjer="0D08
\mathchardef\rpsmjer="0D09
\mathchardef\jkon="3D13
\mathchardef\preth="3D34
\mathchardef\mj="3D36
\mathchardef\mave="3D37
\mathchardef\sljed="3D3C
\mathchardef\vj="3D3E
\mathchardef\vema="3D3F
\mathchardef\trokutic="0D4D
\mathchardef\punitrok="0D4E
\mathchardef\cul="3D62
\mathchardef\ksd="3D63
\mathchardef\mmmanje="3D6E
\mathchardef\vvvece="3D6F
\mathchardef\lgkut="0D70
\mathchardef\dgkut="0D71
\mathchardef\dirsuma="2D75
\mathchardef\ldkut="0D78
\mathchardef\ddkut="0D79
\mathchardef\ominus="0D7F

\mathchardef\hkprec="0E7D
\mathchardef\hrprec="0E7E

\let\N=\mN
\let\R=\mR

\def\Nnul{{\mathbb N}_0}

\def\Rd{{{\mathbb R}^{d}}}

\let\cul=\kul

\def\cl{{\sf Cl\thinspace}}

\def\Re{{\sf Re}\,} 
\def\dom{{\sf dom\thinspace}}

\def\LL#1{{{\rm L}}}

\def\Lb#1{{{\rm L}^\infty(#1)}}

\def\Ld#1{{{\rm L}^{2}(#1)}}

\def\Lb#1{{{\rm L}^{\infty}(#1)}}

\def\supp{{\rm supp\,}}

\def\str{\longrightarrow}

\chardef\sS="19

\def\putover#1{\mathop{\vbox{\ialign{##\crcr\noalign{\kern0pt}
             $\hfil\displaystyle{#1}\hfil$\crcr}}}\limits}    

\def\HmsO{{{\rm H}^{-s}(\Omega)}}

\DeclareMathOperator{\divs}{{\rm div}^s}
\DeclareMathOperator{\dive}{{\rm div}}

\DeclareMathOperator{\grads}{\nabla^s\hspace{-0.05cm}}
\DeclareMathOperator{\Ds}{{\rm D}^s}
\DeclareMathOperator{\DD}{{\rm D}}
\DeclareMathOperator{\ran}{ran}
\def\Mab{\lM(\alpha,\beta;\Omega)}
\def\MabR{\lM(\alpha,\beta;\Rd)}
\def\MabU{\lM(\alpha,\beta;U)}
\def\MabRc{\lM(\alpha,\beta;\Rd\setminus\cl\Omega)}
\def\MabA{\lM(\alpha,\beta;\Omega, \mA_0)}
\def\HnsO{{\rm H}_0^s(\Omega)}

\def \lMsymR{\lM^{\textnormal{sym}}(\alpha,\beta;\Rd)}
\def \lMsymU{\lM^{\textnormal{sym}}(\alpha,\beta;U)}

\author{Andreas Buchinger\thanks{TU Hamburg, Germany, e-mail: andreas.buchinger@tuhh.de}\;\orcidlink{0009-0004-4203-5874}, Krešimir Burazin\thanks{School of Applied Mathematics and Informatics, University of Osijek, Croatia, e-mail: kburazin@mathos.hr}\;\orcidlink{0000-0001-6713-7560}, Ivana Crnjac\thanks{School of Applied Mathematics and Informatics, University of Osijek, Croatia, e-mail: icrnjac@mathos.hr}\;\orcidlink{0000-0002-0435-7184}, Marko Erceg\thanks{Department of Mathematics, Faculty of Science, University of Zagreb, Croatia, e-mail: maerceg@math.hr}\;\orcidlink{0000-0002-4077-9253},\\
Maja Jolić\thanks{School of Applied Mathematics and Informatics, University of Osijek, Croatia, e-mail: mjolic@mathos.hr}\;\orcidlink{0000-0002-9633-3031}, Marcus Waurick\thanks{TU Bergakademie Freiberg, Germany, e-mail: marcus.waurick@math.tu-freiberg.de}\;\orcidlink{0000-0003-4498-3574}$\,\,$
}

\title{Characterisation of homogenisation for nonlocal diffusion by local topologies
}
\date{\today}

\begin{document}

\maketitle

\begin{abstract}
We consider fractional variants of divergence form problems with highly oscillatory local coefficients. We characterise the convergence of these coefficients by means of classical $H$-convergence covering the local behaviour of the fractional divergence form problem and weak-$\ast$ convergence on the complement caused by the nonlocality of the differential operators. The results are further described in the light of nonlocal $H$-convergence as introduced in [Waurick, Calc Var PDEs, 57, 2018] and certain Schur topologies. Applications to symmetric coefficients and a homogenisation problem for a fractional heat type equation are provided.

\end{abstract}

\noindent\textbf{MSC:} Primary: 35B27, 35R11, Secondary: 35B40, 74S40

\noindent \textbf{Keywords:} nonlocal $H$-convergence, $H$-convergence, $H^s$-convergence, $G$-convergence, $G^s$-convergence, fractional derivatives, divergence form problems, homogenisation, Schur topology, Riesz potential

\section{Introduction}\label{sec:intro}

In the modelling of phenomena in mathematical physics it has been apparent that nonlocal equations, that is, equations that take into account the behaviour of particles in a certain region and not only at a single point, provide more accurate descriptions in various cases. A prominent example in electromagnetism is the so-called nonlocal response theory; see \cite[Chapter 10]{K11}. 
In this theory, the material coefficients are represented by convolution kernels, whereas the associated differential operators remain classical.
In contrast, nonlocal equations can also be formulated using fractional differential operators, such as the fractional Laplacian, which replace classical derivatives with nonlocal ones. For a probabilistic interpretation, applications to a wide range of physical systems, and the numerical treatment of such operators, we refer to \cite{Bucur16, DGV24}. Furthermore, fractional derivative terms have been shown to improve predictions for certain physical phenomena; specifically, models involving fractional derivatives often outperform traditional ans\"atze by requiring fewer terms to achieve greater accuracy \cite{B08}.

While several fractional derivatives are commonly used in the one-dimensional setting and the theory is well-standardised \cite{Mainardi, SKM93}, the situation in higher dimensions is significantly more diverse (see \cite{Silh19} and the references therein). Recently, however, the Riesz fractional gradient was rigorously established in \cite{ShSp, Silh19}. By identifying it as a canonical higher-dimensional generalisation, this work has sparked significant interest in the framework, as seen in \cite{BCMC20, CCM, CSjfa19, KS, Silh22}. Our contribution follows this approach.

In nonlocal equations modelling physical phenomena the coefficients  might be highly oscillatory in order to take material variations appropriately into account. In these cases, compared to classical equations, computations are more involved not only due to the nonlocal nature of the equations but also due to the variations of the coefficients on a smaller scale. Hence, the set-up asks for what is called homogenisation, i.e., to find a somewhat optimal approximation of highly oscillatory coefficients by non-oscillatory ones. 
The majority of the literature addresses the homogenization of L\'evy-type operators, where the oscillations are contained within a scalar kernel describing the jumps of the underlying processes (see, e.g., \cite{AM15, BE21, BRS17, KPZ19, SVW21}). The techniques employed range from probabilistic methods to those tailored for partial differential equations. Notably, the work \cite{BRS17} (see also \cite{BE21}) represents the first attempt to adapt the notion of $H$-convergence to nonlocal problems, which is the central theme of our contribution.
In 2018, a more comprehensive theory for local differential operators and nonlocal coefficients has been found which operator-theoretically captures this quest of finding non-oscillatory coefficients. Indeed, in \cite{Wau18}, the notion of nonlocal $H$-convergence has been coined, which, if restricted to local coefficients, coincides with classical homogenisation approaches that are in turn covered by $H$-convergence. It is the purpose of the present text to study homogenisation problems for nonlocal differential equations with local coefficients and fractional derivative operators, building upon the theory introduced in \cite{CCM}.
In this overview, we start by recalling the notion of classical local $H$-convergence.

To ensure a clear presentation of the upcoming content, we fix the following \textbf{default assumptions} and notation, which remain in effect throughout the paper unless specified otherwise.
\begin{itemize}
    \item \textbf{Parameters.} $d\geq 1$ (the dimension of the underlying space $\R^d$), $s\in \langle 0,1\rangle $ (i.e., $0<s<1$, the order of the fractional differential operators), $0<\alpha\leq \beta$ (bounds for the coefficients). 
    \item \textbf{Domain $\Omega$.} In sections 1--4, we require that $\Omega\subseteq\R^d$ is open and bounded. 
    In sections 5--8, we work under the assumption that $\Omega\subseteq \R^d$ is open and bounded, and that the Lebesgue measure of its boundary is equal to zero. 
\end{itemize}

Classically, non-periodic local homogenisation problems employ the notion of $H$-conver\-gence, which dates back to the works of Murat and Tartar in the 1970s; standard references for this topic include \cite{Allaire, Tartar}. To begin, for an open set $U\subseteq\mathbb{R}^d$, we define the set of admissible matrix-valued coefficients:
\[
\begin{aligned}
\MabU\coloneqq \Bigl\{\mA \in L^\infty(U;\R^{d\times d}) : \; & \mA(x)\xi\cdot\xi \geq \alpha|\xi|^2 , \\ 
& \mA^{-1}(x)\xi\cdot\xi \geq \frac{1}{\beta}|\xi|^2,\xi\in\Rd,\text{ a.e. }x\in U \Bigr\} \,.
\end{aligned}
\]
\begin{definition}\label{def:Hc} Let $(\mA_n)_n$, $\mA$ in $\Mab$. Then $(\mA_n)_n$ is said to \textbf{$H$-converge} to $\mA$ ($\mA_n \stackrel{H}{\longrightarrow} \mA$), if for all $f\in {\rm H}^{-1}(\Omega)$ and $u_n \in {\rm H}_0^1(\Omega)$ such that
\[
    -\dive (\mA_n\nabla u_n) = f,
\]
we obtain for $u\in {\rm H}^1_0(\Omega)$, satisfying
\[
  -\dive (\mA\nabla u) = f,
\]
$u_n \xrightharpoonup{ {\rm H}^1_0(\Omega)} u$ and $\mA_n \nabla u_n\xrightharpoonup{ \Ld\Omega^d} \mA \nabla u$.
\end{definition}
The main properties of $H$-convergence are compactness and metrisability (see Theorem \ref{thmLocalHisMetrAndComp}), which illustrate why it is often considered the natural topology on coefficients for these problems.

At first glance, the definition above appears slightly more restrictive than the original one, as we require the parameters $\alpha$ and $\beta$ to be the same for both the sequence and the limit (cf.~\cite[Definition 6.4]{Tartar}). However, due to the compactness of the set of admissible coefficients with respect to $H$-convergence and the uniqueness of the $H$-limit, one can impose such a constraint without loss of generality. This approach is consistent with \cite[Definition 1.2.15]{Allaire}. Similar considerations apply to $G$-convergence (see Section \ref{sec:Gconv}) and to Definition \ref{def:Hsc} (when compared to \cite{CCM}).

In \cite{CCM}, homogenisation problems with local coefficients but nonlocal, fractional, differential operators have been considered. Also for this set-up, the term `nonlocal $H$-convergence' was used. This time, however, in order to describe convergence of sequences of highly oscillatory local coefficients---the nonlocality relates rather to the problems than to the coefficients. In order to distinguish these two apparently different notions of nonlocal $H$-convergence, for the latter, we shall use $H^s$-convergence instead. The definition is as follows. For brevity, we will defer commentary on the unique existence of solutions for the upcoming equations, as well as the precise definition of the operators and spaces, to Section \ref{sec:prelim}.
\begin{definition}\label{def:Hsc} Let $(\mA_n)_n$, $\mA$ in $\MabR$. Then $(\mA_n)_n$ is said to \textbf{$H^s$-converge} to $\mA$ ($\mA_n \stackrel{H^s}{\longrightarrow} \mA$), if for all $f\in {\rm H}^{-s}(\Omega)$ and $u_n \in {\rm H}_0^s(\Omega)$ such that
\[
    -\divs (\mA_n\grads u_n) = f,
\]we obtain for $u\in {\rm H}_0^s(\Omega)$, satisfying
\[
  -\divs (\mA\grads u) = f,
\]
$u_n \xrightharpoonup{{\rm H}_0^s(\Omega)} u$ and $\mA_n \grads u_n\xrightharpoonup{{\rm L}^2(\R^d)^d} \mA \grads u$.
\end{definition} Note that due to the nonlocality of the fractional differential operators, the matrix function sequences need to be defined on the whole of $\R^d$ instead of $\Omega$, only. In consequence, the fluxes, $\mA_n \grads u_n$ weakly converge in ${\rm L}^2(\R^d)^d$ instead of ${\rm L}^2(\Omega)^d$ in the local case.

Since the coefficients are defined on $\Rd$, the role of the domain $\Omega$ is implicit when we state that $(\mA_n)_n$ $H^s$-converges to $\mA$. When we want to emphasise the dependence on $\Omega$, we shall write that $(\mA_n)_n$ $H^s$-converges to $\mA$ in $\Omega$. Unless otherwise specified, all such convergences in this paper are understood to be with respect to this fixed domain $\Omega$.

In \cite{CCM}, their main arguments relate $H^s$-convergence to classical $H$-convergence and, for symmetric coefficients, to $\Gamma$-convergence. It turns out that the main difference between $H^s$-convergence and $H$-convergence is that even if considered on an open bounded set $\Omega$, $H^s$-convergence appears to also have certain implications on the coefficient sequence on the complement of $\Omega$. Anticipating this fact, the authors of \cite{CCM} focused on sequences that are fixed outside $\Omega$, which led to the introduction of
\[
   \MabA\coloneqq \{\mA \in \MabR: \mA|_{\R^d\setminus\Omega} = \mA_0|_{\R^d\setminus\Omega} \}
\]
 for some $\mA_0\in   \MabR$.
Then, even though not explicitly stated, the arguments presented in \cite[Thm. 3.1 and Lem. 3.3]{CCM} suffice to obtain the following (at least for $d\geq 2$): For any  $(\mA_n)_n$, $\mA$ in $ \MabA$, one has
\begin{equation}\label{th:ccm1}\tag{$\ast$}
 \mA_n   \stackrel{H^s}{\longrightarrow}  \mA
\iff  \mA_n|_\Omega \stackrel{H}{\longrightarrow} \mA|_\Omega \,.
\end{equation}
In particular, by sequential compactness of $H$-convergence, as corollary, the authors of \cite{CCM} obtain sequential compactness of $H^s$-convergence. Indeed, in their proof of \cite[Theorem 3.1]{CCM}, they show that $H$-convergence yields $H^s$-convergence and then later in Lemma 3.3 that the $H^s$-limit is unique. This is sufficient to obtain \eqref{th:ccm1} (we refer to Section \ref{sec:MabA} for the precise argument).

In the present paper, under rather mild conditions on $\Omega$, we shall fully characterise $H^s$-convergence without further restrictions on the coefficient sequence other than standard a priori bounds. In fact, we will provide a proof of the following statement.
\begin{theorem}\label{thm:intro-Hs-MabR}
Let $\Omega$ have a boundary of measure zero, $d\geq 1$, $s\in \langle 0,1\rangle$. Let $(\mA_n)_n$, $\mA$ in $ \MabR$. Then
\begin{equation}\label{H^s-conv-gen}
 \mA_n   \stackrel{H^s}{\longrightarrow}  \mA
\iff  \mA_n|_\Omega \stackrel{H}{\longrightarrow} \mA|_\Omega \ \text{ and } \
\mA_n|_{\R^d\setminus \cl\Omega}
\xrightharpoonup{\ * \ }
\mA|_{\R^d\setminus \cl\Omega} \,, 
\end{equation}
where $\xrightharpoonup{*}$ denotes the weak-* convergence in ${\rm L}^\infty(\Rd\setminus\cl\Omega;\R^{d\times d})$.
\end{theorem}
As a consequence, the latter theorem also provides a compactness statement for sequences in $\MabR$ (see also Section \ref{sec:MabR} below). The strategy of the proof is similar as before. At first, one shows that the convergence on the right yields $H^s$-convergence (for this we also invoke a compact embedding statement from \cite{KS}; see Theorem \ref{KS-Lema:2.12}) and then one shows uniqueness of the limit. 


Moreover, we shall draw the connection to nonlocal $H$-convergence as introduced in \cite{Wau18} and further studied in \cite{Wau25}. Among other things, the results in \cite{Wau25} provide an abstract perspective to nonlocal $H$-convergence in that this convergence is identified as a particular Schur topology. Here, the Schur topology, $\tau(\mathcal{H}_0,\mathcal{H}_0^\bot)$, is a topology on a subset of bounded linear operators, $L(\mathcal{H})$, on a Hilbert space $\mathcal{H}$ that can be parametrised by closed subspaces $\mathcal{H}_0$ of $\mathcal{H}$. We refer to Section \ref{sec:schur} for the details. We will show that there are (at least) two different choices for Schur topologies on the coefficient sequence yielding the same metric space induced by $H^s$-convergence. As an application of the latter also taking advantage of the main result in \cite{BSW24} (see also \cite{BFSW25}), we will fully describe a homogenisation problem for a heat type equation with nonlocal differential operators; which is provided next. We note in passing that this answers the question in \cite[Section 6, (3)]{CCM}.  We will postpone well-posedness issues to the corresponding section\footnote{Particularly note that the mentioned operator $\divs$ in the theorem about the heat equation is the ${\rm L}^2$-adjoint of $\grads $.} and denote by $\partial_t$ the weak derivative on ${\rm L}^2(0,T)$ for some fixed $T>0$ with maximal domain---we will reuse its name for the canonical extension to the vector-valued case.
\begin{theorem}\label{thm:intro-evolution-Hs}
Let $\Omega$ have a boundary of measure zero, $d\geq 1$, $s\in \langle 0,1\rangle$. Let $(\mA_n)_n$, $\mA$ in $\MabR$ and assume 
$ \mA_n   \stackrel{H^s}{\longrightarrow}  \mA$; let $f\in {\rm L}^2(0,T; {\rm L}^2(\Omega))$.
Consider the solutions $u_n \colon [0,T\rangle \times \Omega \to \R$ of
\begin{equation}\label{eq:ExampleFracHeatEqAn}
    \partial_t u_n -\divs \mA_n \grads u_n = f,\quad u_n(0)=0.
\end{equation}
Then, $u_n \to u$ in ${\rm L}^2(0,T;{\rm L}^2(\Omega))$, where $u$ satisfies
\begin{equation}\label{eq:ExampleFracHeatEqLimit}
         \partial_t u -\divs \mA \grads u = f,\quad u(0)=0.
\end{equation}
\end{theorem}

Using our insights for the nonsymmetric case, we will also characterise $H^s$-convergence for symmetric coefficients, introducing (the almost self-explanatory) $G^s$-convergence. We gathered our contributions and their comparison to known results in the present manuscript in Figure \ref{fig:intro}, thus completing and extending the picture in \cite{CCM}.

\begin{figure}[ht!]
    \centering
    \begin{tikzpicture}[
        block/.style={
            rectangle, 
            draw, 
            rounded corners=5pt, 
            minimum width=2.5cm, 
            minimum height=1.2cm,
            align=center,
            fill=white
        },
        arrow/.style={
            <->, 
            thick, 
            >=stealth
        },
    ]
    	
    	\coordinate (A) at (-3,-5.5);
    	\coordinate (B) at (10.5,2.5);
    	\draw[rounded corners=20pt, fill=gray!5] (A) rectangle (B);
		
		\node[xshift=-3cm, yshift=-0.8cm, align=center] at (B) {nonlocal $H$-convergence \cite{Wau25}\\ (Thm.~\ref{thm:NonlocalTopologieFracGradPlusOrtho})};

		\coordinate (C) at ([shift={(0.5cm, 0.5cm)}]A);
		\coordinate (D) at ([shift={(-0.5cm, -1.5cm)}]B);
		\coordinate (E) at ([shift={(0.5cm, 0.5cm)}]C);
		\coordinate (F) at ([shift={(-0.5cm, -3cm)}]D);
		
		\draw[dashed, fill=gray!15, rounded corners=10pt] (C) rectangle (D);

		\draw[dotted, thick, fill=gray!30, rounded corners=10pt] (E) rectangle (F);

        \node at (0,-0.5) [block, label={above:$\MabR$}] (1-1) {$H^s$-convergence\\ (Def.~\ref{def:Hsc})};
        \node[block, right=2cm of 1-1] (1-2) {$H$-convergence on $\Omega$\\ +\\ weak-$\ast$ convergence in $\mathbb{R}^d\setminus\cl{\Omega}$};

		\node[block, below=1.8cm of 1-1, label={above:$\MabA$}] (2-1) {$H^s$-convergence\\ (Def.~\ref{def:Hsc})};
		\node[block, right=2cm of 2-1] (2-2) {$H$-convergence\\ (Def.~\ref{def:Hc})};
	
        \draw[arrow] (1-1) -- node[midway, above, font=\footnotesize] {Thm.~\ref{thm:intro-Hs-MabR}} (1-2);
        \draw[arrow] (2-1) -- node[midway, above, font=\footnotesize] {\eqref{th:ccm1}} (2-2);

    \end{tikzpicture}
    \caption{
    Characterisation of $H^s$-convergence for the classes $\MabA$ and $\MabR$. The former is a subset of the latter, consisting of operators with coefficients fixed outside $\Omega$. Here, weak-$\ast$ convergence is understood in the sense of the ${\rm L}^\infty(\mathbb{R}^d\setminus\cl\Omega)^{d\times d}$ space (middle box). In a broader sense, $H^s$-convergence can be interpreted within the framework of nonlocal $H$-convergence (largest box).}
    \label{fig:intro}
\end{figure}

The present paper is organised as follows. In Section \ref{sec:prelim}, we gather some elementary properties for fractional differential operators as well as fractional Sobolev spaces and recall a technical tool useful for homogenisation problems namely a fractional version of compensated compactness. For convenience of the reader and to be self-contained at least in the major aspects of homogenisation theory, we provide the proof. Moreover, in this part of the paper, we provide well-posedness results for the fractional differential problems encountered in this work. Section \ref{sec:MabA} details the main findings of \cite{CCM} in the present context and particularly elaborates on the characterisation result that has been obtained at least implicitly there. In Section \ref{sec:1d}, we focus on the one-dimensional setting. Similar to the local case, the one-dimensional setting is also special for $H^s$-convergence. A prerequisite for this part is established in the previous section, where we strengthen the results of \cite{CCM} to allow for $d=1$.
The subsequent Section \ref{sec:MabR} contains one of our main results, the characterisation of $H^s$-convergence in terms of $H$-convergence and weak-$*$ convergence. We gather some additional properties of $H^s$-convergence in Section \ref{sec:properties} including metrisability, energy convergence and the validity of a lemma eventually leading to a convergence statement for the transpose sequence. Section \ref{sec:schur} describes $H^s$-convergence by means of two different Schur topologies using the main lemma from the previous section and provides the application to a fractional heat type equation. The main body of the manuscript is concluded with Section \ref{sec:Gconv}, where we characterise $H^s$-convergence by $G^s$-convergence for symmetric families.

\section{Preliminaries}\label{sec:prelim}
Fractional calculus is a relatively recent area of research that has attracted growing interest due to its ability to model nonlocal phenomena. Despite rapid progress, the field has not yet converged to a unified theoretical framework, nor does it possess standard references. In this work, we develop the analysis under the most general assumptions possible, with particular emphasis on the one-dimensional case ($d=1$), which exhibits several distinctive features compared to higher-dimensional settings. In the following, we use standard notions for function spaces \cite{Adams}. Moreover,  we use the abbreviation $X(\Rd)$ for $X(\Rd;\R)$ (where $X$ here represents any function space that will be used in the manuscript). Additionally, we sometimes use $X(\Rd)^m$ to denote $X(\Rd;\R^m)$ and $X(\Rd)^{m\times m}$ to denote $X(\Rd;\R^{m\times m})$.
Only in Section~\ref{sec:schur}, the generic codomain will be $\mathbb{C}$.
For two topological vector spaces $X$ and $Y$, we say that $X$ is embedded in $Y$, denoted by $X\hookrightarrow Y$, if there exists a linear, continuous and injective map $i:X\to Y$.
\subsection{Fractional operators}

Following the approach given in \cite{MSP99,Silh19}, we define the space 
$$
\lT(\Rd;\R^m)=\{\varphi\in {\rm C}^{\infty}(\Rd;\R^m): \partial^{\alpha}\varphi\in {\rm L}^1(\Rd;\R^m) \cap {\rm C}_0(\Rd;\R^m), \text{ for every } \alpha\in \Nnul^d\},
$$
endowed with the topology given by the countable family of norms $|\cdot|_k$, $k\in\Nnul$,
$$
|f|_k=\max\Bigl\{\|\partial^{\alpha}f\|_{{\rm L}^1(\Rd;\R^m)}, \|\partial^{\alpha}f\|_{{\rm L}^{\infty}(\Rd;\R^m)}: \alpha\in\Nnul^d,\ 0\leq|\alpha|\leq k\Bigr\}. 
$$

If by $\lD(\Rd;\R^m)$ we denote the space of test functions, i.e., space ${\rm C}_c^{\infty}(\Rd;\R^m)$ with strict inductive limit topology, and by $\lS(\Rd;\R^m)$ the Schwartz space, then we have the following dense embeddings
$$
\lD(\Rd;\R^m) \overset{\rm dense}{\hookrightarrow} \lS(\Rd;\R^m) \overset{\rm dense}{\hookrightarrow} \lT(\Rd;\R^m) \,.
$$
By the property of the transpose operator, we get these embeddings for the dual spaces:
\begin{equation}\label{eq:emb_lTprime}
\lT'(\Rd;\R^m) \hookrightarrow \lS'(\Rd;\R^m) \hookrightarrow \lD'(\Rd;\R^m) \,.
\end{equation}

Let us briefly discuss the differences between the aforementioned dual spaces. For simplicity, we restrict our attention to the scalar case $m=1$.
It is well-known that for any locally integrable function $f\in {\rm L}^1_{\rm loc}(\R^d)$, the linear functional defined by
$$
\lD(\Rd) \ni \varphi\mapsto \int_\Rd f(x)\varphi(x) \,dx
$$
belongs to the space of distributions $\lD'(\Rd)$. This functional can be extended to an element of $\lT'(\Rd)$ if and only if $f\in {\rm L}^1(\Rd)+{\rm L}^\infty(\Rd)$, i.e., there exist $f_1\in  {\rm L}^1(\Rd)$ and $f_\infty\in  {\rm L}^\infty(\Rd)$ such that $f=f_1+f_\infty$, cf.~\cite[Section 6]{Silh19}.
Having the above standard identification in mind, we then have
$$
{\rm L}^1(\Rd)+{\rm L}^\infty(\Rd) 
    \hookrightarrow\lT'(\Rd) \,.
$$

After fixing the necessary notation, let us introduce the definition of the (Riesz) fractional gradient and fractional divergence following \cite[Definition 2.1]{Silh19} (see also \cite[Section 2]{CSjfa19}
and \cite{ShSp}). 

\begin{definition}\label{def:grads-divs} Let $s\in\langle 0,1\rangle$, $\varphi\in \lT(\Rd)$ and $\psi\in \lT(\Rd;\Rd)$. The \textbf{fractional gradient of order $s$} of $\varphi$, and the \textbf{fractional divergence of order $s$} of $\psi$ are, respectively, defined by
\begin{equation}\label{eq:grads}
\grads\varphi(x) = \mu_{s}\int\limits_{\Rd}\frac{(\varphi(x)-\varphi(y))(x-y)}{|x-y|^{d+s+1}}\,dy, \quad x\in\Rd    
\end{equation}
and
\begin{equation}\label{eq:divs}
\divs\psi(x)= \mu_{s}\int\limits_{\Rd}\frac{(\psi(x)-\psi(y))\cdot(x-y)}{|x-y|^{d+s+1}}\,dy, \quad x\in\Rd,
\end{equation}
where $\mu_{s}=2^s\pi^{-d/2}\frac{\Gamma((d+s+1)/2)}{\Gamma((1-s)/2)}$ and $\Gamma$ is the Gamma function,
while $x\cdot y$ denotes the standard Euclidean scalar product on $\Rd$.
\end{definition}

Besides the above definition, these fractional differential operators can be defined in several other equivalent ways, a point we partially address shortly.
For simplicity of exposition, sometimes we will use the notation $\nabla^1=\nabla$ and ${\rm div}^1=\dive$, 
where $\nabla$ and $\dive$ are classical (local) differential operators. 

The above formulae \eqref{eq:grads} and \eqref{eq:divs} are well-defined for bounded Lipschitz functions on $\R^d$ \cite[Lemma 2.3]{ComiSt-I}. However, we choose to focus first on functions from $\lT$ since then the mappings
\begin{equation*}
\grads : \lT(\R^d)\to \lT(\R^d;\R^d) \quad \hbox{and} \quad 
    \divs : \lT(\R^d;\R^d)\to \lT(\R^d)
\end{equation*}
are well-defined and continuous (with respect to the above locally convex topology), cf.~\cite[Theorem 4.3 and Proposition 5.2]{Silh19}. 
Moreover, in the reference it is shown that any continuous operator $\lT(\R^d)\to\lT(\R^d;\R^d)$ which is translationally and rotationally invariant, and $s$-homogeneous, is a multiple of $\grads$. 
Here, the $s$-homogeneity is to be understood in the following way:
\begin{equation*}\label{eq:alpha-hom}
(\forall\varphi\in\lT(\Rd))(\forall \lambda>0)
    \quad \grads\varphi_\lambda(x) 
    = \lambda^s \grads \varphi(\lambda x) \;,
    \quad x\in\Rd \,,
\end{equation*}
where $\varphi_\lambda:=\varphi(\lambda\,\cdot)$.
An analogous property is valid for $\divs$.
Therefore, the above choice for fractional differential operators can be considered canonical. Nevertheless, they have attracted substantial interest only in recent years.



Having the following fractional integration by parts formula in mind,
\begin{equation}\label{eq:frac_int_parts}
 \int\limits_{\Rd} f(x) \divs g(x)\,dx 
    = -\int\limits_{\Rd} \grads g(x) \cdot f(x)\,dx \,, \quad 
    f\in \lT(\Rd), \, g\in \lT(\Rd;\Rd) \,,
\end{equation}
we define the weak notions of fractional operators, 
as given in \cite[Definition 6.1]{Silh19}.

\begin{definition}\label{def:grads-divs-distr} Let $F\in\lT'(\Rd)$ and $G\in\lT'(\Rd;\Rd)$. Then, $\grads F\in\lT'(\Rd;\Rd)$ and $\divs G\in\lT'(\Rd)$ are given by the following relations
$$
{}_{\lT'(\Rd;\Rd)}\langle\grads F, f\rangle_{\lT(\Rd;\Rd)} = -{}_{\lT'(\Rd)}\langle F, \divs f\rangle_{\lT(\Rd)}, \quad f\in\lT(\Rd;\Rd).
$$
and 
\begin{equation*}
  {}_{\lT'(\Rd)}\langle\divs G, g\rangle_{\lT(\Rd)} = -{}_{\lT'(\Rd;\Rd)}\langle G, \grads g\rangle_{\lT(\Rd;\Rd)}, \quad g\in\lT(\Rd).  
\end{equation*} 
\end{definition}

By \eqref{eq:frac_int_parts} this definition is clearly an extension of the operators given in Definition \ref{def:grads-divs}, while the continuity of operators $\grads$ and $\divs$ ensures that the weak fractional operators map $\lT'$ into $\lT'$.
Furthermore, the definition is consistent with the previous comment on bounded Lipschitz functions. Indeed, if $F$ is bounded and Lipschitz continuous, then $F\in\lT'(\Rd)$, while \cite[Proposition 2.8]{ComiSt-I} ensures that $\grads F$, i.e., the above weak fractional gradient, is given by \eqref{eq:grads}. An analogous statement holds for the operator $\divs$.

Although the theory of fractional operators $\grads$ and $\divs$ is rather new, the fractional Laplacian operator $(-\Delta)^{\frac s2}$ is well-established in the literature. 
Hence, we do not state its precise definition, but just refer to \cite{MSP99,Silh19} for the definition of $(-\Delta)^{s}$ for $s\in \langle 0,1\rangle$
on $\lT'(\Rd)$
(note that by \cite[Definition 4.5, Theorem 4.7 and Section 6]{Silh19} this definition coincides on the space ${\rm C}^\infty_c(\Rd)$ with the standard one given e.g.~in \cite[Section 3]{Hitchhiker}).

In the local setting ($s=1$) we have $-\dive \nabla =-\Delta$, where $\Delta$ is the classical (local) Laplacian. This important relation is retained for the fractional operators.
\begin{theorem}\label{eq:prelim_divs_grads_deltas}
    For any $s\in \langle 0,1\rangle$ it holds

\begin{itemize}
    \item[(a)]
$
-\divs(\grads F) = (-\Delta)^{s}F \,,
    \quad F\in\lT'(\Rd) \,,
$
\item[(b)] $
(-\Delta)^{\frac{1-s}{2}}\grads F
    = \grads\,(-\Delta)^{\frac{1-s}{2}} F
    = \nabla F \,, \quad 
    F \in \lT'(\Rd).
$
\end{itemize}
\end{theorem}
While the first identity in the above theorem is provided in \cite[Theorem 5.3]{Silh19} for $\lT(\Rd)$, its straightforward extension to $\lT'(\Rd)$ is deduced directly from the duality argument used to define the weak fractional differential operators. 
The version of the second identity in Theorem \ref{eq:prelim_divs_grads_deltas} 
(see \cite[Theorem 5.3]{Silh19} and \cite[(2.8)]{KS}) adjusted for the Sobolev spaces was heavily used in the proof of the compactness result of $H^s$-convergence in \cite[Theorem 3.1]{CCM}.

 One might be tempted to invert the fractional Laplacian in Theorem \ref{eq:prelim_divs_grads_deltas}\,(b) and get: $\grads=\nabla (-\Delta)^{\frac{s-1}{2}}$.
This indeed holds, but only for smooth functions \cite[Section 5]{Silh19}, and it is more commonly expressed in terms of Riesz potentials $I_\alpha$, which, for $\alpha\in \langle 0,d\rangle$ and a measurable function $f:\Rd\to\R$, is defined by (cf.~\cite[Section 6.1.1]{Grafakos})
\begin{equation*}
I_\alpha f(x) := \frac{\mu_{1-\alpha}}{d-\alpha} \int\limits_\Rd \frac{f(y)}{|x-y|^{d-\alpha}} dy \,,
\end{equation*}
where $\mu_{1-\alpha}$ is given in Definition \ref{def:grads-divs}.
Then in \cite[Section 2.3]{CSjfa19} the following is shown:
\begin{theorem}\label{eq:prelim_riesz_grad}
    For any $\varphi\in {\rm C}^\infty_c(\Rd)$ and any $\psi\in {\rm C}^\infty_c(\Rd;\Rd)$
we have
\begin{equation*}
\grads\varphi = \nabla I_{1-s}\varphi = I_{1-s}\nabla \varphi \quad
    \hbox{and} \quad
    \divs \psi = \dive I_{1-s}\psi = I_{1-s}\dive \psi \,,
\end{equation*}
where for vector-valued functions (here $\nabla \varphi$ and $\psi$)
the Riesz potential $I_{1-s}$ is set to act componentwise.
\end{theorem}
Therefore, as discussed in \cite{CSjfa19}, either of these expressions leads to the same notion of fractional differential operators. 

Let us close this discussion by mentioning another equivalent definition of operators $\grads$ and $\divs$ in terms of the Fourier transformation. Either using Theorem \ref{eq:prelim_divs_grads_deltas} or \ref{eq:prelim_riesz_grad}, one easily gets (see \cite[Theorem 1.4]{ShSp} and \cite[Section 4]{Silh22}) the following result. 
\begin{theorem}\label{ft}
    For any $f\in\lT(\Rd)$ and $g\in\lT(\Rd;\Rd)$ we have
\begin{equation*}
\widehat{\grads f}(\xi)=i\xi|\xi|^{s-1}\widehat{f}(\xi) 
    \quad \text{ and } \quad  
    \widehat{\divs g}(\xi)= i|\xi|^{s-1}\xi\cdot\widehat{g}(\xi) \,,
\end{equation*} 
where $\hat{u}$ denotes the Fourier transform of $u$, given by
$\hat{u}(\xi)= \frac{1}{(2\pi )^{d/2}}\int_\Rd e^{-i x\cdot\xi} u(x) dx$ for all $
\xi\in \Rd$.
For vector-valued functions, the Fourier transform is used componentwise. 
\end{theorem}

To conclude, all three options addressed here, namely, Definition \ref{def:grads-divs}, Theorem \ref{eq:prelim_riesz_grad}, and Theorem \ref{ft}, yield equivalent definitions of fractional differential operators (cf.~\cite{CSjfa19, ShSp}).

\subsection{Fractional Sobolev spaces}\label{subsec:fracSob}

For the analysis of fractional diffusion problems and $H^s$-convergence, we shall need fractional Sobolev spaces.
In this paper, we focus exclusively on the case $p=2$, where ${\rm H}^s(\Rd)={\rm W}^{s,2}(\Rd)$. Here, the former denotes the Bessel potential spaces and the latter the Sobolev–Slobodeckij spaces (see \cite[Section 1.3.1]{Gris85}).
We will continue to use ${\rm H}^s(\Rd)$ to denote these spaces and use any of two equivalent norms. 

Using Theorem \ref{ft} it is an easy exercise to see that 
the norm 
$$
  \|u\|_{s}:=\left(\|u\|^2_{\Ld\Rd} + \|\grads u\|^2_{\Ld\Rd}\right)^{1/2} \,,
    \quad u\in {\rm C}^\infty_c(\Rd) \,,
$$
is equivalent with the norm of ${\rm H}^s(\Rd)$ for functions in ${\rm C}^\infty_c(\Rd)$ (here it is more convenient to use the norm of the Bessel potential spaces which, by the Plancherel theorem, reads: 
$\varphi\mapsto \|(1+|\xi|^2)^{\frac s2}\hat \varphi\|_{\Ld\Rd}$). 
Hence, since ${\rm C}^\infty_c(\Rd)$ is dense in ${\rm H}^s(\Rd)$, 
we have that the above norm $\|\cdot\|_{s}$, when extended to ${\rm H}^s(\Rd)$ via Definition \ref{def:grads-divs-distr}, is an equivalent norm on ${\rm H}^s(\Rd)$, i.e.~$\cl_{\|\cdot\|_{s}}{\rm C}^\infty_c(\Rd)={\rm H}^s(\Rd)$.
This is in fact rigorously proven in \cite[Theorem 1.7]{ShSp} for any integrability exponent $p\in \langle 1,\infty \rangle$. 

\begin{remark}
It is important to notice that the first identity in Theorem  \ref{ft} 
can be extended for $f\in {\rm H}^s(\Rd)$. 
\end{remark}

\begin{remark}
Let us make a short historical comment on the recent development of the fractional Sobolev spaces used here. 
In \cite{ShSp}, for $p\in \langle 1,\infty\rangle$, the following norm 
\begin{equation*}
\|\varphi\|_{s,p} := \Bigl(\|\varphi\|_{{\rm L}^p(\Rd)}^p + \|\grads \varphi\|_{{\rm L}^p(\Rd;\Rd)}^p\Bigr)^{\frac{1}{p}} \,,
\quad \varphi\in {\rm C}^\infty_c(\Rd) \,,
\end{equation*}
is introduced and then ${\rm X}^{s,p}(\Rd)$ is defined as the completion of the normed space 
$$
({\rm C}^\infty_c(\Rd),\|\cdot\|_{s,p}) \,.
$$
Then  it was shown  (\cite[Theorem 1.7]{ShSp}) that ${\rm X}^{s,p}(\Rd)={\rm H}^{s,p}(\Rd)$, where ${\rm H}^{s,p}(\Rd)$ is the Bessel potential space with the integrability exponent $p$ (cf.~\cite[Section 1.3.1]{Gris85}). 

On the other hand, in \cite{CSjfa19} another space 
is introduced:
\begin{equation*}
{\rm S}^{s,p}(\Rd) := \Bigl\{ f\in {\rm L}^p(\Rd) : \grads f\in {\rm L}^p(\Rd;\Rd)\Bigr\} \,,
\end{equation*}
where here $\grads$ denotes the weak derivative in the sense of Definition \ref{def:grads-divs-distr}.
Then for any $p\in [1,\infty]$ the space $({\rm S}^{s,p}(\Rd),\|\cdot\|_{s,p})$ is a Banach space \cite[Proposition 3.20]{CSjfa19} (where again we consider the extension of $\grads$ in the definition of 
$\|\cdot\|_{s,p}$).

In order to prove that ${\rm X}^{s,p}(\Rd)$ and ${\rm S}^{s,p}(\Rd)$ coincide, it is sufficient to show that ${\rm C}^\infty_c(\Rd)$ is dense in ${\rm S}^{s,p}(\Rd)$.
This was done in \cite[Theorem 3.23]{CSjfa19} for $p=1$
and then extended for $p\in [1,\infty\rangle$ in \cite[Theorem A.1]{BCCS22}
(see also \cite[Theorem 2.7]{KS}). Of course, for $p=\infty$ the claim does not hold. 

Therefore, both variants of fractional spaces coincide with Bessel potential spaces. 
\end{remark}

Since in this work we are interested in the study of the suitable boundary value problems on $\Omega\subseteq\Rd$, associated to the introduced fractional operators, we will primarily focus on a closed subspace of ${\rm H}^s(\Rd)$ which contains an information about the boundary condition on $\Omega$.
More precisely, following \cite{CCM}, for $U\subseteq\Rd$ open
we define
\begin{equation*}
{\rm H}^s_0(U) := \cl_{\|\cdot\|_s} {\rm C}^\infty_c(U) \,,
\end{equation*} 
i.e.~the closure of ${\rm C}^\infty_c(U)$ in ${\rm H}^s(\Rd)$.
Of course, ${\rm H}^s_0(\Rd)={\rm H}^s(\Rd)$.
The space ${\rm H}^s_0(U)$ is a separable Hilbert space, with inner product
\begin{equation}\label{eq:HnsO_inner_product}
\langle u,v\rangle_{{\rm H}^s_0(U)}=\int\limits_{U}uv\,dx + \int\limits_{\Rd}\grads u\cdot \grads v\,dx \,,
\end{equation}
where we have used the fact that functions in ${\rm H}^s_0(U)$ are equal to zero a.e.~in $\Rd\setminus U$.
Of course, for any $\varphi\in{\rm C}^\infty_c(U)$ the mapping 
${\rm H}^s(\Rd)\ni u\mapsto \varphi u\in{\rm H}^s_0(U)$ is linear and continuous.

Moreover, we denote the dual of ${\rm H}^s_0(U)$  by ${\rm H}^{-s}(U)$, i.e.,
\begin{equation*}
{\rm H}^{-s}(U) := \left({\rm H}^s_0(U)\right)' \,.
\end{equation*}
Then the following embeddings hold
    $$
    \lD(U)\overset{\rm dense}{\hookrightarrow} {\rm H}^s_0(U) \overset{\rm dense}{\hookrightarrow}\Ld{U} \overset{\rm dense}{\hookrightarrow} {\rm H}^{-s}(U) \hookrightarrow \lD'(U) \,.
    $$
If $U=\Rd$, we can include also the spaces $\lT$ and $\lT'$ to get:
$$
\lD(\Rd)\overset{\rm dense}{\hookrightarrow} \lT(\Rd) \overset{\rm dense}{\hookrightarrow} {\rm H}^s(\Rd) \overset{\rm dense}{\hookrightarrow}\Ld\Rd
 \overset{\rm dense}{\hookrightarrow} {\rm H}^{-s}(\Rd) \hookrightarrow \lT'(\Rd) \hookrightarrow \lD'(\Rd),
$$
where $\Ld\Rd$ is considered as the pivot space, as usual.

\begin{remark}\label{rem:Hs_weak_conv}
From \eqref{eq:HnsO_inner_product} it is clear that a sequence $(u_n)_n$ in ${\rm H}^s(\Rd)$ weakly converges to $u\in  {\rm H}^s(\Rd)$ if and only if
$u_n\xrightharpoonup{\Ld\Rd} u$ and $\grads u_n\xrightharpoonup{\Ld{\Rd;\Rd}} \grads u$.
\end{remark}

Let us briefly comment on some properties of (weak) fractional operators $\grads$ and $\divs$ related to the above fractional Sobolev spaces. 
It is easy to see that the restriction of $\grads$, given in Definition \ref{def:grads-divs-distr}, to ${\rm H}^s(\Rd)$ is a continuous linear operator from  ${\rm H}^s(\Rd)$ to $\Ld{\Rd;\Rd}$.
Of course, the same applies if we replace ${\rm H}^s(\Rd)$ by
${\rm H}_0^s(\Omega)$.

On the other hand, the restriction of $\divs$ to $\Ld{\Rd;\Rd}$
is a continuous linear operator from $\Ld{\Rd;\Rd}$ to 
${\rm H}^{-s}(\Rd)$.
Furthermore, since $\HnsO\subseteq{\rm H}^s(\Rd)$ implies ${\rm H}^{-s}(\Rd)\subseteq\HmsO$, we can define $\divs_{\Omega}:{\rm L}^2(\Rd;\Rd)\to\HmsO$ by the following relation
\begin{equation}\label{eq:divOm}
{}_{\HmsO}\langle \divs_{\Omega} G, g\rangle_{\HnsO} = -\int\limits_{\Rd} G(x)\cdot \grads g(x)\,dx, \quad g\in\HnsO.
\end{equation}
Now, for $G\in\Ld{\Rd;\Rd}$
$$
(\divs G)|_{\HnsO} = \divs_{\Omega}G.
$$
In the sequel we will use $\divs$ to denote both operators. However, we want to emphasise that although $\divs_{\Omega}G = \divs_{\Omega}F$ is equivalent to $(\divs G)|_{\HnsO} =(\divs F)|_{\HnsO}$, it does not imply that $\divs G=\divs F$ on ${\rm H}^s(\Rd)$. Note here that the inclusion ${\rm H}^{-s}(\Rd)\subseteq\HmsO$ is not an embedding since $\HnsO$ is not dense in ${\rm H}^{s}(\Rd)$ in general.
This distinction will be important for some comments that we will make in Section \ref{sec:1d}.

\begin{remark}\label{rem:HnsO_different}
Let us relate the space $\HnsO$ to the existing theoretical framework in the literature.

Recall that the Sobolev-Slobodeckij norm on an open set $U\subseteq\Rd$ is given by 
$\|u\|_{{\rm W}^{s,2}(U)}:=(\|u\|_{\Ld U}^2 + [u]_{s,U}^2)^\frac{1}{2}$,
where the term
\begin{equation*}
[u]_{s,U} := \left(\int\limits_U\int\limits_U 
    \frac{|u(x)-u(y)|^2}{|x-y|^{d+2s}} dx dy\right)^\frac{1}{2} 
\end{equation*}
is the so-called Gagliardo or Slobodeckij seminorm of $u$.
Given the equivalence of the norms $\|\cdot\|_{{\rm W}^{s,2}(\Rd)}$ and $\|\cdot\|_s$, utilising $\|\cdot\|_{{\rm W}^{s,2}(\Rd)}$ instead of $\|\cdot\|_s$ in the definition of $\HnsO$ yields the same space.
In particular, $\HnsO$ coincides with the space $\widetilde{\rm W}^{s,2}_0(\Omega)$ given in \cite[Definition 2.1]{BLP14}
and ${\rm H}^s_\Omega(\Rd)$ in \cite{CCRS21}.
Consequently, we can deduce conclusions about our space $\HnsO$ from the corresponding properties of the spaces discussed above, a relationship we will leverage in the subsequent analysis.

It is worth mentioning that in general we only have the following strict inclusions: 
\begin{equation*}
\begin{aligned}
\HnsO &\subsetneq {\rm W}^{s,2}_0(\Omega) \\
\HnsO &\subsetneq \bigl\{u\in\Ld\Rd : u=0 \ a.e. 
    \ \hbox{in} \ \Rd\setminus\Omega, \ 
    [u]_{s,\Rd}<\infty \bigr\} \,,
\end{aligned}    
\end{equation*}
where ${\rm W}^{s,2}_0(\Omega)$ is defined as the closure of ${\rm C}^\infty_c(\Omega)$
with respect to the weaker norm
$$
\|u\|_{{\rm W}^{s,2}(\Omega)}=(\|u\|_{\Ld \Omega}^2 + [u]_{s,\Omega}^2)^\frac{1}{2} \,,
$$
while the space on the right-hand side of the second inclusion coincides with the spaces ${\rm S}^{s,2}_0(\Omega)$ and ${\rm H}^{s,2}_0(\Omega)$ as used in \cite{KS} and \cite{BCMC20}, respectively.
However, for sufficiently smooth $\Omega$ (e.g.~Lipschitz), 
the equality holds in both inclusions (see e.g.~\cite[Appendix B]{BLP14}, \cite[Appendix]{CCRS21} and \cite[Section 1.3]{Gris85}).
\end{remark}

In order to prove some results on $H^s$-convergence, we shall need additional properties of fractional Sobolev spaces. For their compactness result regarding $H^s$-convergence, the authors of \cite[Theorem 3.1]{CCM} relied on several properties of the fractional Sobolev spaces $\HnsO$ and ${\rm H}^s(\Rd)$ listed in their preliminary Section 2. Specifically, certain conclusions required the continuity of the Riesz potential for ${\rm L}^p$ spaces, which imposed the requirement $1-s < \frac{d}{2}$. However, since $s \in \langle 0,1\rangle $, this condition is automatically met for $d \ge 2$, which explains the restriction used in \cite{CCM}. Since our aim is to address phenomena appearing in the one-dimensional setting (see Section $\text{\ref{sec:1d}}$), we shall now concisely show that all the necessary properties still hold for $d=1$. This immediately implies that \cite[Theorem 3.1]{CCM} can be generalised to the regime $d \ge 1$.

Namely, the arguments provided for Proposition 2.6 (continuity of the Riesz potential), Proposition 2.8 (Poincar\'e inequality) and Proposition 2.9 (Rellich--Kondrašov theorem) from \cite[Section 2]{CCM} require the condition $d\geq 2$. We now present a single theorem that encompasses the version of these statements valid also for $d=1$.
For a clearer presentation we use $X\lesssim_a Y$ to indicate that there exists a constant $C>0$ that depend only on $a$ such that $X\leq C Y$. 

\begin{theorem}\label{thm:properties_d=1}
Let $s\in\langle 0,1\rangle,$ $d\in\mathbb N$, $\Omega\subseteq\Rd$ open and bounded.
\begin{enumerate}
\item[(a)]{\bf(Continuity of the Riesz potential)} $I_{1-s}:\HnsO\to {\rm H}^1_{\rm loc}(\Rd)$ is linear and continuous. 
Moreover, for any $u\in \HnsO$ we have 
$\nabla(I_{1-s}u) = \grads u$ a.e. in $\Omega$.

If $s>1-\frac{d}{2}$, then the statement still holds if we replace the space $\HnsO$ by ${\rm H}^s(\Rd)$.

\item[(b)] {\bf(Poincar\' e inequality)} For any $u\in\HnsO$ we have:
\begin{equation}\label{poincare}
\|u\|_{\Ld\Omega} \lesssim_{d,s,\Omega} \|\grads u\|_{\Ld{\Rd;\Rd}} \,.
\end{equation}

\item[(c)] {\bf (Rellich--Kondrašov theorem)} $\HnsO$ is compactly embedded into $\Ld\Omega$.
\end{enumerate}
\end{theorem}

\begin{proof}
\begin{enumerate}
\item[(a)] The case $s>1-\frac{d}{2}$ is covered in \cite[Proposition 2.6]{CCM} (in fact the stronger version with the space ${\rm H}^s(\Rd)$ is shown).
Let us then consider the case $s\leq 1-\frac{d}{2}$. Note that since $s\in \langle 0,1\rangle$, this is possible only if $d=1$. 
Hence, the case $s\leq1-\frac{d}{2}$ reduces to $s\in\langle \left . 0,\frac{1}{2}\right ]$, $d=1$, and we shall use the notation $\Ds$ instead of $\grads$ (just to indicate that the object is scalar-valued). 
For $\varphi\in {\rm C}^\infty_c(\Omega)$ by  Theorem \ref{eq:prelim_riesz_grad} 
we have $\|\DD I_{1-s}\varphi\|_{\Ld\R} = \|\Ds\varphi\|_{\Ld\R}$. 
Thus, by linearity of the Riesz potential it is left to estimate the ${\rm L}^2$-norm of $I_{1-s}\varphi$.
Let us define $\frac{1}{p}:= 1-\frac{s}{2}$. Since $s\in \langle \left .0,\frac{1}{2}\right ]$,
we have $p\in \langle \left .1,\frac{4}{3}\right ]$ and $\frac{1}{p}>1-s$. Hence, we can apply the standard continuity result for the Riesz potential (known as the Hardy--Littlewood--Sobolev theorem on fractional integration) (see e.g.~\cite[Theorem 6.1.3]{Grafakos} and \cite[Section 4.2, Theorem 2.1]{Mizuta}):
$$
\|I_{1-s}\varphi\|_{{\rm L}^{q}(\R)} \lesssim_{s} \|\varphi\|_{{\rm L}^p(\R)} \,, 
$$
where $\frac{1}{q}=\frac{1}{p}-(1-s)=\frac{s}{2}$.
Next, take $O\subseteq \R$ open and bounded. 
Then using the facts that $p<2<q$ and $\supp\varphi\subseteq\Omega$, we obtain
$$
\|I_{1-s}\varphi\|_{{\rm L}^{2}(O)} \lesssim_{s,O}  
    \|I_{1-s}\varphi\|_{{\rm L}^{q}(O)}
    \lesssim_s  \|\varphi\|_{{\rm L}^p(\R)}
    = \|\varphi\|_{{\rm L}^p(\Omega)} \lesssim_{s,\Omega}\|\varphi\|_{\Ld{\Omega}}\,.
$$

Hence, for any $\varphi\in {\rm C}^\infty_c(\Omega)$ and any $O\subseteq \R$ open and bounded, it holds
$$
\|I_{1-s}\varphi\|_{{\rm H}^1(O)} \lesssim_{s,O,\Omega} \|\varphi\|_s \,.
$$
Since ${\rm C}^\infty_c(\Omega)$ is dense in $\HnsO$, the claim follows. 
\item[(b)]
Via a standard density argument, it suffices to consider $u \in {\rm C}^\infty_c(\Omega)$. The result then is a consequence of the standard Poincar\'e inequality in terms of the Gagliardo semi-norm (see Remark \ref{rem:HnsO_different} and \cite[Lemma 2.4]{BLP14}):
$$
\|u\|_{\Ld\Omega} \lesssim_{d,s,\Omega} [u]_{s,\Rd} \;, \quad 
    u\in {\rm C}^\infty_c(\Omega) \,.
$$
Indeed, one just need to use the following relations (\cite[Proposition 2.8]{ABSS25}):
$$
[u]_{s,\Rd}^2 = C_{d,s} \int_\Rd u (-\Delta)^s u \, dx
    = C_{d,s} \int_\Rd u \divs \grads u \, dx
    = C_{d,s} \int_\Rd  |\grads u|^2 \, dx \,,
$$
where the first equality is a standard one (see e.g.~\cite[Proposition 3.6]{Hitchhiker}), the second is due to Theorem \ref{eq:prelim_divs_grads_deltas}\,(a), 
while the last one follows from the fractional integration by parts formula \eqref{eq:frac_int_parts}.
Since the value of the positive constant $C_{d,s}$ is irrelevant for our purposes, for simplicity we keep it in an implicit form (note that it can be expressed using the constant $\mu_s$ in Definition \ref{def:grads-divs}); we also refer to \cite[Theorem 2.2]{BCMC20}.
\item[(c)] 
By virtue of Remark \ref{rem:HnsO_different}, the claim follows immediately from \cite[Theorem 2.7]{BLP14}; alternatively, \cite[Theorem 2.3]{BCMC20} provides a more direct proof.
\end{enumerate}
\end{proof}


In conclusion, we state two results that will be fundamental for the analysis in Section \ref{sec:MabR}, where the condition that the coefficients are fixed outside $\Omega$ will be omitted.
The first one concerns compactness of the fractional gradient on the complement of $\Omega$: When $s=1$, we get $\nabla u=0$ a.e.~in $\Rd\setminus\cl\Omega\eqqcolon (\cl\Omega)^c$
for any $u\in {\rm H}_0^1(\Omega)$. Hence, the mapping $u\mapsto \nabla u|_{(\cl\Omega)^c}$ is the zero operator. 
On the other hand, for $s\in \langle 0,1\rangle$ and $u\in\HnsO$, we do not have any information on the support of $\grads u$. However, although the mapping  
$u\mapsto  \grads u|_{(\cl\Omega)^c}$ is no longer trivial, we know that it is compact.

\begin{theorem}[{\cite[Lemma 2.12]{KS}}]\label{KS-Lema:2.12} Let $s\in\langle 0,1\rangle$. For any open set $\Omega_1\subseteq\Rd$ such that $\Omega\Subset\Omega_1$  the mapping 
\begin{equation}\label{comp-on-Om^c}
     \HnsO\ni u\mapsto \grads u|_{\Omega_1^c}\in {\Ld{\Omega_1^c}}  
\end{equation}
is compact.
\end{theorem}

The second result is an extension of the property from \cite[Lemma 4.3]{KS}.
\begin{lemma}\label{lem:grads=lambda}
    Let $\Omega\subseteq\Rd$ be open and bounded. For any $x^0\in\Omega\cup(\cl\Omega)^c$ and $\lambda\in\R^d$, there exists $\varphi\in {\rm C}_c^{\infty}(\Omega)$ such that $\grads\varphi(x^0)=\lambda$.
\end{lemma}
\begin{proof} Without loss of generality, let $\Omega=B(0,1)$ and $\lambda=\beta e_1$, $\beta\in\R$.
    For $x^0\in\Omega$, the statement is proved in \cite[Lemma 4.3]{KS}. 
    Let $x^0 \in (\cl\Omega)^c$ and consider non-trivial functions $\psi, \theta \in {\rm C}^{\infty}_c(\langle -1,1\rangle)$ supported in $\left\langle-\frac{1}{\sqrt{d}},-\frac{1}{2\sqrt{d}}\right\rangle \cup \left\langle\frac{1}{2\sqrt{d}},\frac{1}{\sqrt{d}}\right\rangle$. We further require that $\theta \geq 0$, and that $\psi$ is an odd function satisfying $\psi(x) \geq 0$ on $\langle 0, 1 \rangle$. With these functions, we define $\varphi$ as
    $$
    \varphi(x):= \frac{|x^0-x|^{d+s+1}}{|x|^{d+s+1}}\psi(x_1)\theta(x_2)\cdots\theta(x_d), \quad x\in \Rd.
    $$
    Obviously $\varphi(x^0)=0$, while from the properties of $\psi$ and $\theta$, and the fact that for $x\in B( 0,1)$ we have $|x^0-x|>|x^0|-1>0$, it follows that $\varphi\in {\rm C}^{\infty}_c(B(0,1))$. Then the first component of the fractional gradient of $\varphi$ at $x^0$ reads
    \begin{align*}
     (\grads\varphi(x^0))_1 &= \mu_{s}\int\limits_{\Rd}\frac{(\varphi(x^0)-\varphi(y))(x^0_1-y_1)}{|x^0-y|^{d+s+1}}\,dy \\
    &= -\mu_{s}\int\limits_{\Rd}\frac{\psi(y_1)\theta(y_2)\cdots\theta(y_d)(x^0_1-y_1)}{|y|^{d+s+1}}\,dy\\
    &= -\mu_{s}x^0_1\int\limits_{\Rd}\frac{\psi(y_1)\theta(y_2)\cdots\theta(y_d)}{|y|^{d+s+1}}\,dy + \mu_{s}\int\limits_{\Rd}\frac{y_1\psi(y_1)\theta(y_2)\cdots\theta(y_d)}{|y|^{d+s+1}}\,dy.  
    \end{align*}
    Since $\psi$ is odd the first integral in the above equation vanishes. As $\psi$ and $\theta$ are not identically zero and $y_1\psi(y_1)$ and $\theta$ are non-negative, the second integral yields
    $$
    (\grads\varphi(x^0))_1= 2\mu_{s}\int\limits_{\{\Rd:y_1>0\}}\frac{y_1\psi(y_1)\theta(y_2)\cdots\theta(y_d)}{|y|^{d+s+1}}\,dy \eqqcolon \gamma >0.
    $$
    Similarly, for $j\in \{2,\ldots,d\}$, we have
    $$
    (\grads\varphi(x^0))_j= -\mu_{s}x^0_j\int\limits_{\Rd}\frac{\psi(y_1)\theta(y_2)\cdots\theta(y_d)}{|y|^{d+s+1}}\,dy + \mu_{s}\int\limits_{\Rd}\frac{y_j\psi(y_1)\theta(y_2)\cdots\theta(y_d)}{|y|^{d+s+1}}\,dy=0,
    $$
    where we used that both integrands are odd functions with respect to $y_1$. Hence, $\grads (\frac{\beta}{\gamma}\varphi)(x^0)=\beta e_1$, so that $\frac{\beta}{\gamma}\varphi$ is the desired function.
\hfill$\blacksquare$

\subsection{Fractional compensated compactness}

In the classical analysis of partial differential equations, the theory of compensated compactness plays a significant role. It dates back to the work of Murat and Tartar in the second half of the twentieth century \cite{Mu1,Mu2,T1,T5}. 
Specifically, they proved the div-curl lemma, a well-known result providing sufficient conditions under which the product of two weakly convergent sequences converges to the product of their limits. Using fractional integration by parts and the fractional Leibniz rule, we are able to prove a special case of the div-curl lemma in the fractional setting, where one sequence is in the form of $\grads v_n$ (hence, curl-free) and another sequence has strongly convergent fractional divergence ($\divs u_n$ strongly converges in $\HmsO$). We refer to a recent survey \cite{Cap-div-rot}, where this property and two more variants of mixed local-nonlocal div-curl lemmas were proved in the loc-spaces setting.

First, recall the Leibniz rule for fractional gradients (the analogue statement holds for fractional divergence), which can be found in the literature in several forms, depending on the considered function spaces, see e.g. \cite[lemmas 2.4,2.5]{ComiSt-I}, \cite[Lemma 3.2]{ComiSt}, \cite[Lemma 2.11]{KS}. Here we state the version which is suitable for our analysis.

\begin{proposition}[Fractional Leibniz rule]\label{Frac-leibniz}
    Let $\varphi\in{\rm C}^{1}_c(\Omega)$ and $u\in{\rm H}^s(\Rd)$. Then, $\varphi u\in\HnsO$ and 
    $$
    \grads(\varphi u)=\varphi\grads u + u\grads\varphi + \grads_{\rm NL}(\varphi,u),
    $$
    where the remainder term $\grads_{\rm NL}(\varphi,u)$ is given by 
    $$
    \grads_{\rm NL}(\varphi,u)(x):=
    \mu_{s}\int\limits_{\Rd}\frac{(\varphi(x)-\varphi(y))(u(x)-u(y))(x-y)}{|x-y|^{d+s+1}}\,dy, \quad  x\in\Rd.
    $$
    Moreover, $\grads_{\rm NL}(\varphi,\cdot):\Ld\Rd\to\Ld{\Rd;\Rd}$ is a linear and continuous map which satisfies 
    $$
    \|\grads_{\rm NL}(\varphi,u)\|_{\Ld{\Rd;\Rd}} \lesssim_{d,s}\|\varphi\|^{1-s}_{{\rm L}^{\infty}(\Omega)}\|\nabla\varphi\|^s_{{\rm L}^{\infty}(\Omega)}\|u\|_{\Ld\Rd}.
    $$

\end{proposition}

\begin{proposition}[Fractional compensated compactness]\label{prop:div-rot} 
    Let $\Omega\subseteq\Rd$ be open and bounded, and let $(u_n)_n$ in $\Ld{\Rd;\Rd}$ and $(v_n)_n$ in ${\rm H}^s(\Rd)$. For some $u\in\Ld{\Rd;\Rd}$ and $v\in{\rm H}^s(\Rd)$ assume
    $$
    v_n\xrightarrow{\Ld\Rd}v, \quad  v_n\xrightharpoonup{{\rm H}^s(\Rd)}v,
    $$
    $$
    u_n\xrightharpoonup{\Ld{\Rd;\Rd}}u, \quad   \divs u_n\xrightarrow{\HmsO}\divs u.
    $$
    Then 
    $$
    u_n\cdot\grads v_n \xrightharpoonup{\lD'(\Omega)} u\cdot\grads v.$$
\end{proposition}

\begin{remark}
Note that by Theorem \ref{thm:properties_d=1}\,(c) and Remark \ref{rem:Hs_weak_conv} the assumptions on $v_n$ can be replaced by a stronger requirement
$$
(v_n)_n \text{ in }\HnsO \quad \text{and} \quad v_n\xrightharpoonup{\HnsO}v.
$$  
\end{remark}

\begin{proof}
    For $\varphi\in{\rm C}^{\infty}_c(\Omega)$, using the fractional Leibniz rule we get
    \begin{eqnarray*}
     \int\limits_{\Omega}(u_n\cdot\grads v_n)\varphi\,dx &=& 
    \int\limits_{\Rd}u_n\cdot(\varphi\grads v_n)\,dx  \\
    &=& \int\limits_{\Rd}u_n\cdot\grads(\varphi v_n)\,dx - \int\limits_{\Rd}(u_n\cdot\grads\varphi) v_n\,dx -
    \int\limits_{\Rd}u_n\cdot\grads_{\rm NL}(\varphi,v_n)\,dx.
    \end{eqnarray*}
    For the first term, we have 
    \begin{equation}\label{konv1}
    \int\limits_{\Rd}u_n\cdot\grads(\varphi v_n)\,dx =-\,{}_{\HmsO}\langle\divs u_n,\varphi v_n\rangle_{\HnsO} \xrightarrow{n\to\infty} 
    -\,{}_{\HmsO}\langle\divs u,\varphi v\rangle_{\HnsO}.
    \end{equation}
    Regarding the second term, since $u_n\xrightharpoonup{\Ld{\Rd;\Rd}}u$, $\grads\varphi\in{\rm L}^{\infty}(\Rd;\Rd)$ and $v_n\xrightarrow{\Ld\Rd}v$, we get 
    \begin{equation}\label{konv2}
    \int\limits_{\Rd}(u_n\cdot\grads\varphi) v_n\,dx\xrightarrow{n\to\infty}\int\limits_{\Rd}(u\cdot\grads\varphi) v\,dx.
    \end{equation}
    Moreover, from Proposition \ref{Frac-leibniz} we have that $v_n\xrightarrow{\Ld{\Rd}}v$ implies $\grads_{\rm NL}(\varphi,v_n)\xrightarrow{\Ld{\Rd;\Rd}}\grads_{\rm NL}(\varphi,v)$, which together with $u_n\xrightharpoonup{\Ld{\Rd;\Rd}}u$ gives 
    \begin{equation}\label{konv3}
    \int\limits_{\Rd}u_n\cdot\grads_{\rm NL}(\varphi,v_n)\,dx\longrightarrow \int\limits_{\Rd}u\cdot\grads_{\rm NL}(\varphi,v)\,dx.
    \end{equation}  
    Using \eqref{konv1}, \eqref{konv2} and \eqref{konv3}, and applying the fractional Leibniz rule once more, this time to $\varphi$ and $v$, we obtain
    $$
    \int\limits_{\Omega}(u_n\cdot\grads v_n)\varphi\,dx \longrightarrow
    \int\limits_{\Omega}(u\cdot\grads v)\varphi\,dx,
    $$
    which concludes the proof.
\end{proof}

We conclude this section by giving an auxiliary result which is an extension (to the fractional setting) of the classical characterisation of the dual Sobolev space. This property will be significant for the analysis of the one-dimensional homogenisation problem.

\begin{lemma}\label{lem:f-g}
   For $f\in \HmsO$, there exists $g\in\Ld{\Rd;\Rd}$ such that $f=\divs g$ in $\HmsO$, and the mapping $f\mapsto g$ is a linear isometry $\HmsO\str\Ld{\Rd;\Rd}$.
\end{lemma}
\begin{proof}
    From the Poincar\'e inequality it follows that $\|\cdot\|_{\HnsO}$ and $\|\grads\cdot\,\|_{\Ld\Rd}$ are equivalent norms on $\HnsO$, implying that $\HnsO$ is a Hilbert space when endowed with the (equivalent) inner product given by
    \begin{equation*}
        \langle v,u\rangle = \int\limits_{\Rd}\grads v\cdot\grads u\,dx, \quad u,v\in\HnsO.
    \end{equation*}
    Furthermore, the Riesz representation theorem implies that for a given $f\in\HmsO$ there exists a unique $v\in\HnsO$ such that 
    \begin{equation}
    {}_{\HmsO}\langle f,u\rangle_{\HnsO} = \int\limits_{\Rd}\grads v\cdot\grads u\,dx, \quad u\in\HnsO,
    \end{equation}
    with $f\mapsto v$ being an isometric isomorphism $\HmsO\str \HnsO$. Having in mind the definition of $\divs$ on $\Ld{\Rd;\Rd}$, we conclude 
    $$
    f=\divs(-\grads v) \quad \text{on } \HnsO.
    $$
    Hence, $g=-\grads v\in\Ld{\Rd;\Rd}$ is the desired function, and $f\mapsto g$ is clearly a linear isometry $\HmsO\str\Ld{\Rd;\Rd}$.
\end{proof}

\subsection{Fractional diffusion: well-posedness}

For  $f\in\HmsO$ we consider the fractional diffusion problem as follows.
\begin{equation}\label{fp}
\left\{
	\begin{array}{rl}
		-\divs (\mA\grads u)=f \quad   &\rm{in}\,\, \Omega, \\
		u=0 \quad &{\rm in}\,\, \Rd\setminus\Omega. \\
	\end{array}
	\right.
\end{equation}
Here, we assume $\mA\in\MabR$, where, let us recall, 
    $\lM(\alpha,\beta;U)$ denotes the class of all matrix-valued functions $\mA:U\to\R^{d\times d}$ satisfying
    \begin{align}
        &\mA(x)\xi\cdot\xi\geq\alpha|\xi|^2, \quad &\text{for every $\xi\in\Rd$ and for a.e. $x\in U$},\label{mA1}\\
        &\mA(x)\xi\cdot\xi\geq\frac{1}{\beta}|\mA(x)\xi|^2, \quad &\text{for every $\xi\in\Rd$ and for a.e. $x\in U$}. \label{mA2}
    \end{align}

We say that $u\in\HnsO$ is a \textbf{weak} or \textbf{variational solution of the problem} (\ref{fp})
if for any $v\in \HnsO$
\begin{equation*}
\int\limits_\Rd \mA(x)\grads u(x) \cdot\grads v(x) \, dx
    = {}_{\HmsO}\langle f, v\rangle_{\HnsO} \,. 
\end{equation*}

\begin{proposition}
Let $\mA\in\MabR$. For any $f\in\HmsO$ there exists a unique weak solution
$u\in \HnsO$ of \eqref{fp}. Moreover, the solution $u$ satisfies the following estimate
\begin{equation}\label{a-pr-est}
\|\grads u\|_{\Ld{\Rd;\Rd}} \leq \frac{1}{\alpha} \|f\|_\HmsO \,.
\end{equation}
\end{proposition}
The well-posedness result presented in \cite[Lemma 2.12]{CCM} is provided for $d \ge 2$
by means of the Lax--Milgram theorem. However, the case $d=1$ is obtained analogously, since the necessary properties have been established for this setting in Subsection \ref{subsec:fracSob} (see Theorem \ref{thm:properties_d=1}). 
A discussion of well-posedness for more general nonlocal problems can be found in \cite{FKV15}. Note that in the present setting and, quite generally also for a nonlinear setting, as it was shown in the main result of \cite{TW14}, the only requirement for well-posedness is a mere closed range inequality as for instance provided by the (fractional) Poincar\'e inequality and an invertibility condition satisfied by the coefficients. The former has been exemplified above and the latter is implied by the coercivity condition on the conductivities.

\section{$H^s$-convergence on $\MabA$}\label{sec:MabA}

In this section, we focus on the homogenisation problems with coefficients from the class $\MabA$, and we elaborate on the results obtained in \cite[Sec. 3]{CCM}. Recall that, for a fixed $\mA_0\in\MabR$, we denote by $\MabA$ the class of all matrix-valued functions $\mA\in\MabR$ satisfying $\mA(x)=\mA_0(x)$, for almost every $x\in\Rd\setminus\Omega$. 

As we have mentioned, the proof  of $H^s$-compactness for this class of coefficients, \cite[Thm. 3.1]{CCM}, relies on the assumption that $d\geq2$. Now, with the properties from Theorem \ref{thm:properties_d=1}, all statements used in the proof of \cite[Thm. 3.1]{CCM} also hold for $d=1$ (see the discussion above Theorem \ref{thm:properties_d=1}). Therefore, the following results are valid for any $d\geq1$.

\begin{theorem}[{\cite[Thm. 3.1, Lem. 3.3]{CCM}}]\label{thm:CCM}
    Let $\mA_0\in\lM(\alpha,\beta;\Rd)$. For any sequence $(\mA_n)_n$ in $\MabA$, there exist a subsequence (which we do not relabel) and $\mA\in\MabA$ such that 
    $$
    \mA_n \text{ $H^s$-converges to } \mA \text{ in } \Omega.
    $$
    Furthermore, the limit matrix function $\mA$ is unique, i.e., if for $\mA,\tilde{\mA}\in\MabA$ it holds that
    $$
    \mA_n \text{ $H^s$-converges to } \mA \text{ in } \Omega 
    $$
    and 
    $$
    \mA_n \text{ $H^s$-converges to } \tilde{\mA} \text{ in } \Omega,
    $$
    then 
    $$
    \mA(x)= \tilde{\mA}(x) \quad \text{for a.e. } x\in\Rd.
    $$
\end{theorem}

The compactness obtained above follows directly from the central theorem of $H$-convergence, which provides both compactness and metrisability. For ease of reference, we summarise these results in the following theorem.
\begin{theorem}
[{\cite[Thm.~1.2.16 and Prop.~1.2.24]{Allaire}}]\label{thmLocalHisMetrAndComp}
   There is a metric $d_H$ on $\Mab$ that induces $H$-convergence. Moreover, $(\Mab,\tau_H)$ is compact, where $\tau_H$ is the topology induced by $d_H$.
\end{theorem}

Now, let us discuss implications of the results presented in Theorem \ref{thm:CCM} and clarify how they imply equivalence between $H$- and $H^s$-convergence, stated in \eqref{th:ccm1}. 
\paragraph{Proof of \eqref{th:ccm1}.}
    From the proof of \cite[Thm. 3.1]{CCM} we see that $H$-convergence implies $H^s$-convergence, i.e., if a sequence $(\mA_n)_n$ in $\Mab$ $H$-converges to $\mA\in\Mab$, then, for every fixed $\mA_0\in\MabR$, the sequence 
    $$\widehat{\mA}_n(x)=\left\{
	\begin{array}{rl}
		\mA_n(x),  & x\in\Omega \\
		\mA_0(x), & x\in\Rd\setminus\Omega\\
	\end{array}
	\right.
    $$
    $H^s$-converges to 
    $$\widehat{\mA}(x)=\left\{
	\begin{array}{rl}
		\mA(x),  & x\in\Omega \\
		\mA_0(x), & x\in\Rd\setminus\Omega.\\
	\end{array}
	\right.
    $$
    
    On the other hand, assume that a sequence $(\mA_n)_n$ in $\MabA$ $H^s$-converges to $\tilde{\mA}\in \MabA$ and let $\mB_n:=\mA_n|_\Omega$. Then,  $(\mB_n)_n$ in $\Mab$ and we can pass to a subsequence $(\mB_{n_k})_{n_k}$ which $H$-converges to $\mB\in\Mab$, say. The discussion above then provides that $\mA_{n_k}$ $H^s$-converges to 
    $\widehat{\mB}$, where $\widehat{\mB}$ is the extension of the matrix function $\mB$ on $\Omega$ to the whole $\Rd$ in the above sense.
    Hence, $\mA_{n_k}$ $H^s$-converges to both $\tilde{\mA}$ and $\widehat{\mB}$, and the uniqueness of the $H^s$-limit implies that $\mB=\tilde{\mA}$ a.e.~on $\Omega$. In particular, the $H$-limit does not depend on the chosen subsequence of $(\mB_n)_n$. Thus, we can conclude that the whole sequence $\mB_n=\mA_n|_\Omega$ $H$-converges to $\tilde{\mA}|_\Omega$. This concludes the proof.
    $\hfill\blacksquare$

\begin{remark}\label{rem:loc-nonloc}
    In order to write the observations from \eqref{th:ccm1} in a more systematic way, let us
    define the mapping 
    $$
    \lR:\MabA\to\Mab, \quad \lR(\mA):=\mA|_\Omega.
    $$
    Then $\lR^{-1}(\mA)=\widehat{\mA}$, where $\widehat{\cdot}$ is defined in the proof of \eqref{th:ccm1}, so $\lR$ is a bijection. Furthermore, the above discussion implies that for any sequence $(\mA_n)_n$ in $\MabA$ we have: 
    $$
    \mA_n \stackrel{H^s}{\longrightarrow} \mA \quad \iff \quad \lR(\mA_n) \stackrel{H}{\longrightarrow} \lR(\mA).
    $$

    Although two notions of convergence are equivalent (in the above sense), the corresponding diffusion problems differ significantly due to their local/nonlocal nature, which influences both the solution and the diffusion process.

Furthermore, consider the metric $d_H$ from Theorem~\ref{thmLocalHisMetrAndComp}. Then the pullback metric given by
$$
\hat{d}_H(\mA,\mB)=d_H\left(\lR(\mA),\lR(\mB)\right) \,, \quad \mA,\mB\in\MabA \,,
$$
is a metric on $\MabA$, compatible with the $H^s$-convergence.
Indeed, 
\begin{align*}
\mA_n\xrightarrow{H^s} \mA &\iff \lR(\mA_n) \xrightarrow{H} \lR(\mA) \\
&\iff \lR(\mA_n) \xrightarrow{d_H} \lR(\mA) \\
&\iff \mA_n \xrightarrow{\hat d_H} \mA \,.
\end{align*}
Thus, the $H^s$-topology on $\MabA$ is metrisable, and $\lR$ is a homeomorphism between the metric spaces $(\MabA,\hat{d}_H)$ and $(\Mab,d_H)$. 
\end{remark}
Having in mind Remark \ref{rem:loc-nonloc}, we deduce that  some properties of $H^s$-convergence follow directly from their $H$-convergence counterparts (see \cite[Sec. 1.2.4]{Allaire}). Moreover, following the construction of the metric $d_H$ from Theorem~\ref{thmLocalHisMetrAndComp} given in \cite[Proposition 1.2.24]{Allaire}, we are able to define another metric on $\MabA$ which relies on the properties of fractional diffusion, and is also compatible with the $H^s$-convergence. Of course, such a metric is topologically equivalent with $\hat{d}_H$ given in Remark \ref{rem:loc-nonloc}.

\begin{proposition}[Locality of $H^s$-convergence]\label{prop:locality}
Let $\mA_0\in\MabR$. Moreover, let $(\mA_n)_n$ and $(\mB_n)_n$ be two sequences in $\MabA$, which $H^s$-converge to $\mA$ and $\mB$, respectively. Then, for every open $\omega\subset\Rd$, it holds
    $$
    (\forall n\in\N) \quad  \mA_n(x)=\mB_n(x) \text{ a.e. on } \omega \quad \Rightarrow 
    \quad \mA(x)=\mB(x) \text{ a.e. on } \omega.
    $$
\end{proposition}

\begin{proposition}[Transpose invariance]  
    Let $\mA_0\in\MabR$. If $(\mA_n)_n$ in $\MabA$ $H^s$-converges to $\mA\in\MabA$, then 
    $$
    \mA_n^{\rm T} \text{ $H^s$-converges to } \mA^{\rm T}.
    $$
\end{proposition}

\begin{proposition}[Energy convergence]\label{prop:energy-conv} 
 Let $\mA_0\in\MabR$ and $(\mA_n)_n$ in $\MabA$, $\mA\in\MabA$ be such that 
    $$\mA_n \text{ $H^s$-converges to } \mA.
    $$
    Then, for any $f\in\HmsO$ the sequence $(u_n)_n$ of solutions to
    \begin{equation}\label{pr_n}
    \left\{
	\begin{array}{rl}
		-\divs (\mA_n\grads u_n)=f  &\rm{in}\,\, \Omega, \\
		u_n=0 &{\rm in}\,\, \Rd\setminus\Omega,\,\\
	\end{array}
	\right.
    \end{equation}
    satisfies
    \begin{equation}\label{energy-1}
    \mA_n\grads u_n \cdot \grads u_n \xrightharpoonup{\lD'(\Omega)} \mA\grads u \cdot \grads u
    \end{equation}
    and 
    \begin{equation}\label{energy-2}
    \int\limits_{\Rd} \mA_n\grads u_n \cdot \grads u_n\, dx \to \int\limits_{\Rd} \mA\grads u \cdot \grads u\, dx,
    \end{equation}
    where $u\in\HnsO$ is the unique weak solution to 
    \begin{equation}\label{pr}
    \left\{
	\begin{array}{rl}
		-\divs (\mA\grads u)=f  &\rm{in}\,\, \Omega, \\
		u=0 &{\rm in}\,\, \Rd\setminus\Omega.\,\\
	\end{array}
	\right.
    \end{equation}
\end{proposition}

\begin{proof}
    Convergence \eqref{energy-1} is a direct consequence of Proposition \ref{prop:div-rot} (Fractional div-curl lemma). In order to prove convergence \eqref{energy-2}, we shall use weak formulations of problems \eqref{pr_n} and \eqref{pr}. From \eqref{pr_n} and the fact that $u_n\xrightharpoonup{\HnsO}u$, we get
    \begin{equation}\label{energy_n}
    \int\limits_{\Rd} \mA_n\grads u_n \cdot \grads u_n\, dx ={}_{\HmsO}\langle f,u_n\rangle_{\HnsO}\to {}_{\HmsO}\langle f,u\rangle_{\HnsO}.
    \end{equation}
    On the other hand, the weak formulation for the problem \eqref{pr} implies
    $$
    \int\limits_{\Rd} \mA\grads u \cdot \grads u\, dx = {}_{\HmsO}\langle f,u\rangle_{\HnsO},
    $$
    which together with \eqref{energy_n} gives the desired result.
\end{proof}

\begin{theorem}[Metrisability]\label{thm:metr}
Let $(f_n)_{n\in\N}$ be a dense countable family in $\HmsO$, $s\in\langle 0,1\rangle$, $0<\alpha\le\beta$. We define a metric on $\MabA$ by
\begin{equation}\label{metric}
    d_s(\mA,\mB)=\sum_{n=1}^\infty 2^{-n}\frac{\|u_n-v_n\|_{\Ld\Omega}+\|\mA\grads u_n-\mB\grads v_n\|_{\HmsO}}{\|f_n\|_\HmsO},
\end{equation}
where $u_n,\,v_n\in \HnsO$ are unique solutions of 
$$
	\left\{
	\begin{array}{rl}
		-\divs (\mA\grads u_n)=f_n  &\rm{in}\,\, \Omega, \\
		u_n=0 &{\rm in}\,\, \Rd\setminus\Omega,\,\\
	\end{array}
	\right.\quad
    	\left\{
	\begin{array}{rl}
		-\divs (\mB\grads v_n)=f_n  &\rm{in}\,\, \Omega, \\
		v_n=0 &{\rm in}\, \,\Rd\setminus\Omega.\,\\
	\end{array}
	\right.\,
$$
Then $d_s$ is indeed a metric on $\MabA$ generating the $H^s$-topology on $\MabA$.
\end{theorem}
Before we turn to the proof of this statement, we quickly provide an elementary lemma comparing (weak) convergences in embedded spaces.
\begin{lemma}\label{lem:weakbddweakconv}
    Let $\mathcal{H}, \mathcal{H}_1$ be Hilbert spaces with $\mathcal{H}_1\hookrightarrow \mathcal{H}$ continuously. If $(v_n)_n$ in $\mathcal{H}_1$ is bounded and (weakly) convergent in $\mathcal{H}$ to some $v\in \mathcal{H}$. Then $v\in \mathcal{H}_1$ and $v_n\rightharpoonup v$ in $\mathcal{H}_1$.
\end{lemma}

\begin{proof}
   As $(v_n)_n$ is bounded in $\mathcal{H}_1$, we find a weakly convergent subsequence with limit $w\in \mathcal{H}_1$. As $\mathcal{H}_1\hookrightarrow \mathcal{H}$ is continuous, it is weakly continuous, and hence this subsequence is weakly convergent in $\mathcal{H}$ to $w$ but also to $v$, by assumption. Hence, $v=w\in \mathcal{H}_1$ and the subsequence principle ensures that $(v_n)_n$ itself is weakly convergent in $\mathcal{H}_1$ to $v$.
\end{proof}

\paragraph{Proof of Theorem \ref{thm:metr}.}
    From the Poincar\'e inequality \eqref{poincare} and the a priori estimate \eqref{a-pr-est}, we get
    \begin{equation}\label{est1}
        \|u_n\|_{\Ld\Omega}+\|\mA\grads u_n\|_{\HmsO}\lesssim_{d,s,\Omega,\beta} \|\grads u_n\|_{\Ld\Rd}\lesssim_{\alpha} \|f_n\|_\HmsO,
    \end{equation}
    and analogous estimate holds for $\mB$ and $v_n$. Hence, we can conclude that $d_s$ is well-defined, i.e., $0\leq d_s(\mA,\mB)<\infty$, for all $\mA,\, \mB\in\MabA$.
    Since symmetry and triangle inequality are trivially satisfied, in order to prove that $d_s$ is a metric, it remains to show that $d_s(\mA, \mB)=0$ implies $\mA=\mB$ a.e. on $\Omega$. (Note that on $\R^d\setminus\Omega$ we have  $\mA=\mB=\mA_0$.)
    
    For an arbitrary $u\in\HnsO$ define $f:=-\divs (\mA\grads u)\in\HmsO$, and denote by $v\in\HnsO$ the unique solution of 
  $$
    \left\{
	\begin{array}{rl}
		-\divs (\mB\grads v)=f   &\rm{in}\, \,\Omega, \\
		v=0 &{\rm in}\, \,\Rd\setminus\Omega.\,\\
	\end{array}
	\right.\,
$$
   Take a subsequence of $(f_{n_k})_{k\in\N}$ such that $f_{n_k}\str f$ in $\HmsO$. Then $u-u_{n_k}$ is the unique solution of  
     $$
    \left\{
	\begin{array}{rl}
		-\divs \bigl(\mA\grads (u-u_{n_k})\bigr)=f-f_{n_k}   &\rm{in}\, \,\Omega, \\
		u-u_{n_k}=0 &{\rm in}\, \,\Rd\setminus\Omega,\,\\
	\end{array}
	\right.\,
$$
and by estimate (\ref{est1}), applied to $u-u_n$, we obtain 
$$
u_{n_k} \xrightarrow{\Ld\Omega} u \quad \text{and} \quad \mA\grads u_{n_k} \xrightarrow{\HmsO} \mA\grads u.
$$
Similarly, 
$$
v_{n_k} \xrightarrow{\Ld\Omega} v \quad \text{and} \quad \mB\grads v_{n_k} \xrightarrow{\HmsO} \mB\grads v.
$$
Since $d_s(\mA,\mB)=0$, it follows that $u_{n_k}=v_{n_k}$ in $\Ld\Omega$ and $\mA\grads u_{n_k}=\mB\grads v_{n_k}$ in $\HmsO$. Therefore, $u=v$ and $\mA\grads u=\mB\grads v$, which further implies
$$
(\mA-\mB)\grads u=0 \quad \,\rm{a.e.\,in }\,\, \Omega. 
$$
Moreover, if we denote by $\tilde\Omega\subseteq \Omega$ the set of all Lebesgue points of $\mA-\mB$, we can conclude that $\tilde\Omega$ is contained in the set of all Lebesgue points of $(\mA-\mB)\grads u$, for any $u\in {\rm C}_c^\infty(\Omega)$. Therefore, $(\mA-\mB)\grads u=0$ holds in $\tilde\Omega$, for any $u\in {\rm C}_c^\infty(\Omega)$.

Let $x_0\in\tilde\Omega$ be arbitrary. By Lemma \ref{lem:grads=lambda}, for any $\lambda\in\Rd$ there exists $u\in {\rm C}_c^\infty(\Omega)$ such that $(\grads u)(x_0)=\lambda.$ Using this we obtain $\mA=\mB$ in $\tilde\Omega.$

It is left to prove that the convergence in this metric is indeed $H^s$-convergence. Let $(\mA_m)_m$ be a sequence in $\MabA$ which $H^s$-converges to $\mA$, and let $u_n,\, u_n^m\in \HnsO$ be the solutions to
$$
-\divs (\mA\grads u_n)=f_n \qquad \hbox{and}
    \qquad -\divs (\mA_m\grads u_n^m)=f_n \,,
$$
respectively. By the definition of $H^s$-convergence, for every $n\in\N$, 
$$
u_n^m \xrightharpoonup{\HnsO} u_n \quad \text{and} \quad 
\mA_m\grads u_n^m \xrightharpoonup{\Ld{\R^d}} \mA\grads u_n, \quad \text{as} \quad m\to\infty.
$$
The above convergences and the compact embedding from Theorem \ref{thm:properties_d=1}\,(c) imply
$$
u_n^m \xrightarrow{\Ld\Omega} u_n \quad \text{and} \quad 
\mA_m\grads u_n^m \xrightarrow{\HmsO} \mA\grads u_n, \quad \text{as} \quad m\to\infty.
$$
This, together with the uniform convergence of the series in the definition of $d_s$, implies $d_s(\mA_m,\mA)\str 0$.

Conversely, let $(\mA_m)_m$ in $\MabA$ and $\mA\in\MabA$ be such that 
$$d_s(\mA_m,\mA)\str 0.$$ 
Let $f\in\HmsO$ and choose a subsequence of $(f_{n_k})$ such that $f_{n_k}\to f$ in $\HmsO$.
Furthermore, let $u,\, u^m,\, u_{n_k},\, u_{n_k}^m\in\HnsO$ denote the solutions to 
\begin{alignat*}{2}
-\divs (\mA\grads u)&=f \,, &\qquad
    -\divs (\mA_m\grads u^m) &=f \,, \\
-\divs (\mA\grads u_{n_k}) &=f_{n_k} \,, &\qquad 
    -\divs (\mA_m\grads u_{n_k}^m)&=f_{n_k} \,,
\end{alignat*}
respectively. Since $\|u_{n_k}-u_{n_k}^m\|_{\Ld\Omega}$ and $\|\mA\grads u_{n_k}-\mA_m\grads u_{n_k}^m\|_{\HmsO}$ appear in the series defining $d_s(\mA,\mA_m)$, we have that for each $n_k$  they converge to 0, i.e., for every $k\in\N$
$$
u_{n_k}^m \xrightarrow{\Ld\Omega} u_{n_k} \quad \text{and} \quad 
\mA_m\grads u_{n_k}^m \xrightarrow{\HmsO} \mA\grads u_{n_k} \quad \text{as} \quad m\to\infty.
$$
Since $(u_{n_k}^m)_m$ and $(\mA_m\grads u_{n_k}^m)_m$ are bounded in $\HnsO$ and $\Ld\Omega$, respectively, by Lemma \ref{lem:weakbddweakconv}, 
\begin{equation}\label{conv-m-n'fixed}
u_{n_k}^m \xrightharpoonup{\HnsO} u_{n_k} \quad \text{and} \quad 
\mA_m\grads u_{n_k}^m \xrightharpoonup{\Ld\Omega} \mA\grads u_{n_k}, \quad \text{as} \quad m\to\infty.
\end{equation}
The first convergence in (\ref{conv-m-n'fixed}) implies $\grads u_{n_k}^m\xrightharpoonup{\Ld{\Rd}}\grads u_{n_k}$, which together with $\mA_m=\mA=\mA_0$ on $\Omega^c=\Rd\setminus\Omega$ gives   $\mA_m\grads u_{n_k}^m \xrightharpoonup{\Ld{\Omega^c}} \mA\grads u_{n_k}$.
Hence, for every $k\in\N$  
\begin{equation}\label{conv-m-n'fixed-Rd} 
\mA_m\grads u_{n_k}^m \xrightharpoonup{\Ld{\Rd}} \mA\grads u_{n_k}, \quad \text{when} \quad m\to\infty.
\end{equation} 
Since $u_{n_k}-u,\, u_{n_k}^m-u^m\in\HnsO$ are solutions of  $-\divs (\mA\grads (u_{n_k}-u))=f_{n_k}-f$ and $-\divs (\mA_m\grads (u_{n_k}^m-u^m))=f_{n_k}-f,$ from estimate (\ref{est1}) and convergence $f_{n_k}\xrightarrow{\HmsO} f$ we have 
\begin{equation}\label{conv-n'}
u_{n_k} \xrightarrow{\HnsO} u \quad \text{and} \quad 
    \mA\grads u_{n_k} \xrightarrow{\Ld{\Rd}} \mA\grads u, \quad \text{as} \quad k\to\infty,
\end{equation}
and, uniformly with respect to $m$,
\begin{equation}\label{conv-n'-mfixed}
u_{n_k}^m \xrightarrow{\HnsO} u^m \quad \text{and} \quad 
    \mA_m\grads u_{n_k}^m \xrightarrow{\Ld{\Rd}} \mA_m\grads u^m, \quad \text{when} \quad k\to\infty.
\end{equation}
Finally, \eqref{conv-m-n'fixed}, \eqref{conv-m-n'fixed-Rd}, \eqref{conv-n'} and \eqref{conv-n'-mfixed} imply that $\mA_m$ $H^s$-converges to $\mA$.
$\hfill\blacksquare$

\section{The one-dimensional case}\label{sec:1d}

The homogenisation in the classical setting, when $s=1$, and one space dimension ($d=1$) is rather simplified as the $H$-limit of a sequence coincides with its harmonic mean. This is mainly due to the fact that in one dimension operators ${\rm div}$ and $\nabla$ coincide, i.e., ${\rm div}=\nabla=\frac{d}{dx}$, and thus, from ${\rm div} F=\frac{dF}{dx}=0$, for $F\in\Ld\Omega$, we conclude $F={\rm const}$.
Indeed, for  $A_n$ in $\Mab$, $\Omega\subseteq\R$ and $f\in{\rm H}^{-1}(\Omega)$,
using the above properties, the one-dimensional, static, diffusion problem
$$
\left\{
	\begin{array}{ll}
		-\frac{d}{dx}\Big(A_n\frac{d}{dx}u_n\Big)=f \quad   &\rm{in}\,\, \Omega, \\
		u_n\in\rm{H}^1_0(\Omega), \\
	\end{array}
	\right.
$$
reduces to 
$$
A_n\frac{d}{dx}u_n=c_n-g(x),
$$
where $g\in\Ld{\Omega}$ is such that $f=\frac{dg}{dx}$ and $(c_n)_n$ is a bounded sequence in $\R$. This further implies strong ${\rm{L}}^2$-convergence of a subsequence $\sigma_n=A_n\frac{d}{dx} u_n$, which is a key step in the deduction that the $H$-limit is the harmonic mean  $\underline{A}$ of $A_n$, i.e., the inverse of the weak-$\ast$ limit of $\frac{1}{A_n}$ (for more details, see \cite[Sec. 1.2.3]{Allaire}).

By Remark \ref{rem:loc-nonloc}, the same result holds for $H^s$-convergence in the one-dimensional case. 
More precisely, for $s\in\langle 0,1\rangle$, let $\Ds$ denote operators $\divs$ and $\grads$ in the one dimensional setting (for $d=1$ they coincide). Then, the corresponding one-dimensional, static, fractional diffusion problem is given by 
\begin{equation}\label{frac-1d}
\left\{
	\begin{array}{rl}
		-\Ds(A_n\Ds u_n)=f \quad   &\rm{in}\,\, \Omega, \\
		u_n=0 \quad &{\rm in}\,\, \R\setminus\Omega, \\
	\end{array}
	\right.
\end{equation}
where  $(A_n)_n$ in $\MabA$. Let $A^s$ denote the $H^s$-limit of the sequence $(A_n)_n$, and let $A$ 
be the $H$-limit of the restricted sequence $(A_n|_{\Omega})_n$ in $\Mab$. Since $H$- and $H^s$-limits coincide a.e.~on $\Omega$, we infer
that $A^s = A = \underline{A}$ a.e.~on $\Omega$. Consequently, the $H^s$-limit $A^s$ is given by
$$A^s(x)=\begin{cases}
    \underline{A}(x), \quad x\in\Omega\\
    A_0(x), \quad x\in\Omega^c \,.
    \end{cases}$$

\begin{remark}
    Although both $H$- and $H^s$-limits (for $d=1$) are equal to $\underline{A}$ on $\Omega$, there are some essential differences between the classical ($s=1$) and the fractional case ($s\in\langle 0,1\rangle$). Specifically, in the fractional case, we are unable to obtain strong ${\rm{L}}^2$-convergence (up to a subsequence) of the fluxes $\sigma_n^s:=A_n\Ds u_n$. We will elaborate on that in more detail.

    For $f\in\HmsO$, Lemma \ref{lem:f-g} implies that there exists $g\in\Ld\R$ such that $f=\Ds g$. Then, problem \eqref{frac-1d} reduces to
    \begin{equation}\label{Ds=0}
    \Ds(A_n\Ds u_n + g)=0 \quad \text{in } \HmsO.
    \end{equation}
    The same is true in the classical ($s=1$) case, just with $\frac{d}{dx}$ instead of $\Ds$. However, in contrast to the classical case, where we can conclude that the expression in the brackets above is a sequence converging to a constant, in the fractional case the relation \eqref{Ds=0} is not sufficient to provide a strong convergence (up to a subsequence) of $(\sigma^s_n)_n$ in $\Ld\R$. More precisely, neither $\Ds F=0$ implies $F=const.$, nor from $\Ds F_n=0$ we can conclude the strong convergence of $F_n$, as we shall see in the next lemma.

\end{remark}

\begin{lemma} Let $\Omega\subseteq\R$ be open and bounded.
    Then there exists a bounded sequence $(F_n)_n$ in $\Ld{\R}\setminus\{0\}$ such that, for every $n\in\N$,  $\Ds F_n=0$ in $\HmsO$, and $(F_n)_n$ does not have a strongly convergent subsequence. Consequently, the operator $\Ds:\Ld{\R}\to\HmsO$ has an infinite-dimensional kernel.
\end{lemma}
\begin{proof}
    Since $\Omega$ is bounded, there exists $M>0$ such that $\Omega\subseteq[-M,M]$. Let $w\in{\rm C}_c^{\infty}(\R\setminus\Omega)$ be such that $\supp w\subseteq \langle M,\infty\rangle$ and $\int_{\R}w(x)\,dx=0$, i.e. $\widehat{w}(0)=0$.
    
    If we define the sequence $w_n(x):=w(x-n)$, $n\in\N$, then $(w_n)_n$ in ${\rm C}_c^{\infty}(\R\setminus\Omega)$
    and $\widehat{w}_n(0)=\widehat{w}(0)=0$, since the Fourier transform of $w_n$ reads
     $\widehat{w}_n(\xi)=e^{-in\xi}\widehat{w}(\xi)$. Moreover, for $n\in\N$ we define $F_n$ by
    $$
    \widehat{F}_n(\xi):=\frac{\widehat{w}_n(\xi)}{i\xi|\xi|^{s-1}}=e^{-in\xi}\widehat{F}(\xi), \quad \widehat{F}(\xi):=\frac{\widehat{w}(\xi)}{i\xi|\xi|^{s-1}}, \quad(\xi\in \R).
    $$
    Let us first prove that $(F_n)_n$ is a bounded sequence in $\Ld{\R}$. 
    Since $|F_n|=|F|$, it is sufficient to prove that $F\in\Ld\R$.
    From $w\in{\rm C}^{\infty}_c(\R\setminus\Omega)$, we have that $\widehat{w}$ is a Schwartz function and, by definition of $w$, $\widehat{w}(0)=0$. This in addition implies
    $\widehat{w}(\xi)=\mathcal{O}(|\xi|)$, when $\xi\to0$, hence there exist $\delta>0$ and $m>0$ such that $|\widehat{w}(\xi)|\leq m|\xi|$, when $0<|\xi|<\delta$. We have
        $$
        \|\widehat{F}\|^2_{\Ld\R}=\int\limits_{\R}\frac{|\widehat{w}(\xi)|^2}{|\xi|^{2s}}\,d\xi=\int\limits_{-\delta}^{\delta}\frac{|\widehat{w}(\xi)|^2}{|\xi|^{2s}}\,d\xi + \int\limits_{\R\setminus(-\delta,\delta)}\frac{|\widehat{w}(\xi)|^2}{|\xi|^{2s}}\,d\xi \,.
        $$
       For the first integral in the above sum we apply the above mentioned estimate of $\widehat{w}$ around the origin: 
        $$
        \int\limits_{-\delta}^{\delta}\frac{|\widehat{w}(\xi)|^2}{|\xi|^{2s}}\,d\xi \leq 2m^2\int\limits_0^{\delta}\xi^{2-2s}\,d\xi= 2m^2\frac{\delta^{3-2s}}{3-2s}<\infty,
        $$
        since $s<1$. Further on, for the second integral, it holds
        $$
        \int\limits_{\R\setminus(-\delta,\delta)}\frac{|\widehat{w}(\xi)|^2}{|\xi|^{2s}}\,d\xi\leq \frac{1}{\delta^{2s}}\int\limits_{\R\setminus (-\delta,\delta)}|\widehat{w}(\xi)|^2\,d\xi\leq \frac{1}{\delta^{2s}}\|\widehat{w}\|^2_{\Ld{\R}} <\infty.
        $$
        By the Plancherel theorem, it follows that $F\in\Ld\R$, implying that $(F_n)$ is a bounded sequence in $\Ld{\R}$, i.e.,
        $\|F_n\|_{\Ld\R}=\|F\|_{\Ld\R}<\infty$.
        
        Let us now prove that $\Ds F_n=0$  in $\HmsO$. Using the Fourier transform property  from Theorem \ref{ft}, we get $\widehat{\Ds F_n}(\xi) = i\xi|\xi|^{s-1}\widehat{F}_n=\widehat{w}_n$ in $\Ld \R$, which implies $\Ds F_n=w_n$ a.e. in $\R$. Now, for every $\varphi\in\HnsO$, since the supports of $\varphi$ and $w_n$ are mutually disjoint, it holds
        $$
        {}_{\HmsO}\langle \Ds F_n,\varphi\rangle_{\HnsO} =\int\limits_{\R} \Ds F_n(x) \varphi(x)\,dx
        =\int\limits_{\R} w_n(x)\varphi(x)\,dx=0.
        $$
        
        It remains to prove that $(F_n)_n$ does not have a strongly convergent subsequence. Notice first that, for any $\varphi\in\Ld\R$, we have that $\widehat{F}\,\overline{\widehat{\varphi}}\in{\rm L}^1(\R)$, and by the Riemann-Lebesgue lemma it holds
        $$
\int\limits_{\R}\widehat{F}_n(\xi)\overline{\widehat{\varphi}}(\xi)\,d\xi=
        \int\limits_{\R}e^{-in\xi}\widehat{F}(\xi)\overline{\widehat{\varphi}}(\xi)\,d\xi=\mathcal{F}\{\widehat{F}\,\overline{\widehat{\varphi}}\}(n) \xrightarrow{n\to\infty} 0,
        $$
        which implies $\widehat{F}_n\xrightharpoonup{\Ld{\R}}0$, $n\to\infty$.

        If $(F_{n_k})_k $ is a subsequence such that, for some $F^\ast\in\Ld\R$,  $F_{n_k}\xrightarrow{\Ld\R}F^*$, then, from $\widehat{F}_n\xrightharpoonup{\Ld{\R}}0$ it follows that $F^*=0$, which is in contradiction with $\|F_n\|_{\Ld\R}=\|F\|_{\Ld\R}\neq0$.
        
     The existence of a bounded sequence $(F_n)_n$ in $\ker(\Ds)$ which does not have a strongly convergent subsequence implies that the unit ball in $\ker(\Ds)$ is not a compact set. Hence, $\ker(\Ds)$ has infinite dimension.  
\end{proof}

It is worth emphasising that the above lemma provides an intrinsic property of the (nonlocal) fractional derivative. It indicates that the operator $\Ds:\Ld{\R}\to\HmsO$ is more complex, compared to the classical derivative $\DD=\frac{d}{dx}:\Ld{\Omega}\to{\rm H}^{-1}(\Omega)$ where we have $\dim(\ker(\DD))=1$. This illustrates why, in the one-dimensional fractional setting, there are some difficulties in deriving strong convergence of the fluxes $\sigma^s_n=A_n\Ds u_n$.

\begin{remark} 
  If we consider $\Ds:\Ld\R\to{\rm H}^{-s}(\R)$, then from the assumptions $F\in\Ld\R$ and $\Ds F=0$ in ${\rm H}^{-s}(\R)$, using the interaction with the Fourier transformation  from Theorem \ref{ft}, we obtain $F=0$ a.e. in $\R$. Hence, when $\Omega=\R$, we have the same conclusion as in the local, $s=1$ case.

\end{remark}

\section{$H^s$-convergence on $\lM(\alpha,\beta;\Rd)$}\label{sec:MabR}

Although in the analysis of fractional diffusion problems we are looking for the solutions in $\HnsO$, the nonlocality of the fractional gradient requires for the coefficients to be defined on the whole space $\Rd$. In \cite{CCM}, the authors proposed a way to overcome this difficulty by considering the class $\MabA$, with coefficients fixed outside of $\Omega$. As we can see from previous results, this approach is effective and, by connecting fractional and classical operators as well as fractional and classical Sobolev spaces, it gives a solid foundation for the development of the homogenisation theory in the fractional setting. 

In this section, we expand this approach by considering the coefficients from $\MabR$, without imposing any additional conditions on $\R^d\setminus\Omega$. 

\begin{remark}\label{rem:intro-decomp}
Here, and in the rest of the paper, we assume that the domain $\Omega\subseteq\R^d$ is open and bounded, and
that the Lebesgue measure of its boundary is equal to zero. This assumption on $\Omega$ is crucial for our arguments, as it guarantees that the space $\Ld{\R^d}$ can be isomorphically decomposed as the direct sum:
$$
\Ld\Omega \oplus \Ld{\Rd\setminus\cl\Omega} \,.
$$
Moreover, we also have that $\cl_{\Ld\Rd}{\rm C}^\infty_c(\Rd\setminus\partial\Omega) = \Ld{\R^d}$, which is in fact equivalent to the previous decomposition. Both identities are readily verified.
The decomposition above permits us to effectively divide the convergence analysis on $\mathbb{R}^d$ into two independent parts: $\Omega$ and $\Rd\setminus\cl\Omega$.

Let us close this remark by noting that we could equivalently require that $\Omega\subseteq\R^d$ is open, bounded and Jordan measurable \cite[Prop. 10.5.1]{Lebl25}.
In particular, the standard assumption that $\Omega$ is a Lipschitz domain implies that its boundary has zero measure.
\end{remark}

A key ingredient for obtaining $H^s$-compactness in $\MabR$ is Theorem \ref{KS-Lema:2.12}.
Roughly speaking, the main idea is to consider separately cases when we are on $\Omega$ and on $\Omega^{\text{c}}$, and to combine obtained results for $\MabA$, strong convergence in the complement \eqref{comp-on-Om^c} and properties of weak-$\ast$ convergence.

\begin{remark}\label{rem:nap-bounds}
    As already mentioned, for $U\subseteq\Rd$ open, the bounds in the definition of $\MabU$ are chosen in such a way that this set is closed under the $H$-limit. Let us verify that $\MabU$ is also closed under the arithmetic and harmonic mean of a sequence, i.e., for a sequence $(\mA_n)_n$ in $\MabU$ we have $\overline{\mA},\underline{\mA}\in\MabU$, where $\overline{\mA}$ is a weak-$\ast$ limit of (a subsequence of) $(\mA_n)$ and $\underline{\mA}$ is the inverse of a weak-$\ast$ limit of (a subsequence of) $\mA_n^{-1}$.
    Since $\mA\in \MabU$ is equivalent to $\mA^{-1}\in \lM\left (\frac{1}{\beta}, \frac{1}{\alpha}; U\right )$, it suffices to show that inequalities \eqref{mA1} and \eqref{mA2} are preserved under the weak-$\ast$ limit.
    
    The expression on the left-hand side of the first inequality (\ref{mA1}) in the definition of $\MabU$ is linear in $\mA$, and thus it is also weakly-$\ast$ continuous, meaning that this inequality is preserved on the weak-$\ast$ limit. 
    
    Regarding the second inequality (\ref{mA2}), the left-hand side is similar as above, while the expression on the right-hand side is slightly more complicated, as it is a composition of linear functions and a quadratic function $u\mapsto u^2$, $\Lb \Rd \str \Lb\Rd$. However, we can use weakly-$\ast$ sequential lower semi-continuity of this function on $\Lb \Rd$ to conclude that this inequality is also preserved in the weak-$\ast$ limit.
\end{remark}

Now we are able to prove $H^s$-compactness for the class $\MabR$, and to give (unique) representation of the $H^s$-limit via $H$-limit on $\Omega$ and weak-$\ast$ limit on $\Rd\setminus\cl \Omega$.
We first establish the uniqueness of the $H^s$-limit.

\begin{proposition}\label{prop:Hs_MabR_unique}
The $H^s$-limit within $\MabR$ is unique. More precisely, 
if a sequence $(\mA_n)_n$ in $\MabR$ $H^s$-converges to both $\mA$ and $\mA_\infty$ (in $\Omega$), where $\mA,\mA_\infty\in\MabR$, then $\mA=\mA_\infty$ a.e.~in $\Rd$.
\end{proposition}

\begin{proof}
Let $(\mA_n)_n$ be a sequence in $\MabR$ such that it 
$H^s$-converges to $\mA$ and $\mA_\infty$ in $\Omega$, where $\mA,\mA_\infty\in\MabR$.

Following the lines of the proof of \cite[Lemma 3.3]{CCM}, where uniqueness of the $H^s$-limit in $\MabA$ is shown, we obtain 
    $$
    \int\limits_{\Rd} \mA\grads\psi \cdot \Phi\,dx = \int\limits_{\Rd} \mA_{\infty}\grads\psi \cdot \Phi\,dx, \quad \text{for every }\psi\in {\rm C}_c^{\infty}(\Omega) \text{ and } \Phi\in\Ld{\Rd;\Rd},
    $$
    which implies
    $$
    \mA(x)\grads\psi(x) = \mA_{\infty}(x)\grads\psi(x), \quad \text{ for a.e. $x\in\Rd$ and for all $\psi\in {\rm C}_c^{\infty}(\Omega)$}.
    $$
This step differs from the proof of \cite[Lemma 3.3]{CCM} only in the use of Lemma \ref{lem:grads=lambda} (a refined version of \cite[Lemma 4.3]{KS}) to obtain
    $$
    \mA(x) = \mA_{\infty}(x), \quad \text{ for a.e. }x\in\Rd.
    $$
This concludes the proof of uniqueness.
\end{proof}

\paragraph{Proof of Theorem \ref{thm:intro-Hs-MabR}.}
Assume that $\mA_n|_\Omega \stackrel{H}{\longrightarrow} \mA|_\Omega$ and  $\mA_n|_{\R^d\setminus \cl\Omega}
\xrightharpoonup{\ * \ }
\mA|_{\R^d\setminus \cl\Omega}$, and let us show:
\begin{equation}\label{H^s-convergence}
  \mA_n \text{ $H^s$-converges to $\mA$ in $\Omega$}.
\end{equation}

 Let $f\in\HmsO$ and denote by $u_n$ the unique weak solution of
 \begin{equation}\label{fp1}
\left\{
	\begin{array}{rl}
		-\divs (\mA_n\grads u_n)=f \quad   &\rm{in}\,\, \Omega, \\
		u_n=0 \quad &{\rm in}\,\, \Rd\setminus\Omega. \\
	\end{array}
	\right.
\end{equation}
Since $(u_n)_n$ and $(\mA_n\grads u_n)_n$ are bounded in $\HnsO$ and $\Ld{\Rd;\Rd}$, respectively, there exist joint subsequences (which we do not relabel), $u\in\HnsO$ and $\sigma\in\Ld{\Rd;\Rd}$ such that
   $$
   u_n\xrightharpoonup{\HnsO}u \quad \text{and} \quad \mA_n\grads u_n \xrightharpoonup{\Ld{\Rd;\Rd}}\sigma, \quad \text{as } n\to\infty.
   $$
   Define
   $$
   \vg_n:=\mA_n\grads u_n-\mA\grads u, \quad n\in\N.
   $$
   It is enough to prove that $(\vg_n)_n$ weakly converges to zero in $\Ld{\Rd;\Rd}$, as then
   $$
   \sigma = \mA\grads u \quad \text{a.e. in } \Rd,
   $$
   which further gives \eqref{H^s-convergence}.
   
   Since $(\vg_n)_n$ is bounded in  $\Ld{\Rd;\Rd}$, we can test the weak convergence with functions from ${\rm C}_c^\infty(\Rd;\Rd)$. Moreover, the assumptions on $\Omega$ and Remark \ref{rem:intro-decomp} imply that it suffices to consider functions $\varphi=\varphi_1 +\varphi_2$, where $\varphi_1\in {\rm C}_c^\infty(\Omega;\Rd)$ and $\varphi_2\in {\rm C}_c^\infty((\cl\Omega)^{\text{c}};\Rd)$. Then, we have
   $$
   \int_\Rd \vg_n\cdot\varphi\,dx=\int_{\Rd} \vg_n\cdot\varphi_1\, dx + \int_{\Rd} \vg_n\cdot\varphi_2\, dx.
   $$
   Since $\supp\varphi_1\subseteq\Omega$, following the arguments given in the proof of \cite[Theorem 3.1]{CCM}, we obtain
   $$
   \lim_{n\to \infty}\int_\Rd \vg_n\cdot\varphi_1\, dx =  \lim_{n\to \infty}\int_{\Omega} \vg_n\cdot\varphi_1\, dx=0.
   $$
   Furthermore, since $\varphi_2\in {\rm C}_c^\infty((\cl\Omega)^{\text{c}};\Rd)$, there exists open and bounded $\Omega_1\subseteq\Rd$ such that $\Omega\Subset\Omega_1$ and $\supp\varphi_2\subseteq(\cl\Omega_1)^{\text{c}}$. Hence, from Theorem \ref{KS-Lema:2.12} and assumption $\mA_n|_{\Rd\setminus\cl\Omega}\xrightharpoonup{\ \ast\ }\mA|_{\Rd\setminus\cl\Omega}$, it follows
   $$
   \mA_n\grads u_n\xrightharpoonup{\Ld{(\cl\Omega_1)^c;\Rd}} \mA\grads u,
   $$
   which further implies $\vg_n\xrightharpoonup{\Ld{(\cl\Omega_1)^c;\Rd}}0$, and we get
   $$
   \lim_{n\to \infty}\int_\Rd \vg_n\cdot\varphi_2\, dx = \lim_{n\to \infty}\int_{(\cl\Omega_1)^c} \vg_n\cdot\varphi_2\, dx =0\,,
   $$
   completing the argument. 
   

Now, assume that $(\mA_n)_n$ $H^s$-converges to $\mA$ in $\Omega$, and define sequences $(\mB_n)_n$ in $\Mab$ and $(\tilde{\mA}_n)_n$ in $\MabRc$ by 
$$
\mB_n:=\mA_n|_{\Omega} \quad \text{and} \quad \tilde{\mA}_n:=\mA_n|_{\Rd\setminus\cl\Omega}.
$$
From compactness of the $H$-convergence (Theorem \ref{thmLocalHisMetrAndComp}), it follows that there exist $\mB\in \Mab$ and a subsequence $\mB_{n'}$, such that $\mB_{n'} \stackrel{H}{\longrightarrow} \mB$. Furthermore, the Banach-Alaoglu theorem implies compactness of the weak-$\ast$ convergence, which, together with closedness of the weak-$\ast$ limit in $\MabRc$ (Remark \ref{rem:nap-bounds}), gives the existence of $\tilde{\mA}\in \MabRc$ and a subsequence $\tilde{\mA}_{n'}$ (chosen jointly with the subsequence of $\mB_{n}$) such that $\tilde{\mA}_{n'}\xrightharpoonup{\ \ast\ }\tilde{\mA}$.
Thus, the above proved statement ($H$-convergence on $\Omega$ and weak-$\ast$ convergence on $\Rd\setminus\cl\Omega$ imply $H^s$-convergence) and the fact that the measure of the boundary of $\Omega$ is zero imply 
$$
\mA_{n'} \quad \stackrel{H^s}{\longrightarrow} \quad  \hat{\mA} :=\begin{cases}
\mB,\quad \text{on }\Omega\\
 \tilde{\mA},\quad \text{on }(\cl\Omega)^c
 \end{cases}.
$$
Now, from the uniqueness of the $H^s$-limit provided by Proposition \ref{prop:Hs_MabR_unique}, we conclude $\mA=\hat{\mA}$ a.e.~on $\Rd$. In particular, $\mA|_{\Omega}=\mB$ and $\mA|_{\Rd\setminus\cl\Omega}=\tilde{\mA}$. Hence, 
\begin{equation}\label{Hs-weak*-conv}
  \mA_{n'}|_{\Omega}=\mB_{n'}\stackrel{H}{\longrightarrow} \mB= \mA|_{\Omega} \quad  \text{and} \quad \mA_{n'}|_{\Rd\setminus\cl\Omega}=\tilde{\mA}_{n'}\xrightharpoonup{\ \ast\ }\tilde{\mA}=\mA|_{\Rd\setminus\cl\Omega}.  
\end{equation}
Furthermore, since both the $H$-limit and the weak-$\ast$ limit do not depend on the chosen subsequences, we conclude that \eqref{Hs-weak*-conv} holds for the whole sequence $\mA_n$.
\end{proof}

From Theorem \ref{thm:intro-Hs-MabR}, the compactness of the $H$-convergence and the weak-$\ast$ convergence, we have that the compactness of $H^s$-convergence on $\MabR$ follows directly. More precisely, we obtain the following result.

\begin{corollary}\label{compactness_result}
For any $(\mA_n)_{n\in\mathbb N}$ in $\MabR$ there exists a subsequence $(\mA_{n_k})_{k\in\mathbb N} $ and $\mA\in \MabR$ such that $(\mA_{n_k})$ $H^s$-converges to $\mA$.
\end{corollary}

\section{Properties of $H^s$-topology on $\MabR$}\label{sec:properties}
In this chapter, we present several fundamental properties of $H^s$-convergence on $\MabR$. We begin with a \v Zikov type lemma, a fundamental tool in the homogenisation theory, see \cite[Theorem 5.2]{JKO}, where a similar statement was formulated for symmetric coefficients and $H$-convergence.

\begin{proposition} [\v Zikov type Lemma.] \label{prop:ZikovLemmaFrac}
 Let $f\in\HmsO$ and $\mv\in\Ld{\Rd;\Rd}$. 
 Assume that $(\mA_n)_n$ in $\MabR$ is such that 
 $$
 \mA_n^{\rm T} \ H^s\text{-converges to } \mA^{\rm T}\in\MabR,
 $$
 and denote by $(u_n)_n$ the sequence in $\HnsO$ of solutions to 
\begin{equation}\label{diff-pr-Zhikov}
    -\divs(\mA_n(\grads u_n+\mv))=f.
\end{equation}
Then, there exists a subsequence of $(u_n)_n$, which we do not relabel, such that
$$
u_n \xrightharpoonup{\HnsO}u \quad\quad  \text{and} \quad\quad 
\mA_n(\grads u_n+\mv) \xrightharpoonup{\Ld{\Rd;\Rd}}\mA(\grads u+\mv),
$$
where $u\in\HnsO$ is the unique solution to 
$$-\divs(\mA(\grads u+\mv))=f.$$
\end{proposition}
\begin{proof}
    By denoting $\tilde{f}_n=f+\divs(\mA_n\mv)\in\HmsO$, the sequence of problems \eqref{diff-pr-Zhikov} reduces to 
    $$
    -\divs(\mA_n\grads u_n)=\tilde{f}_n.
    $$
    For every $n\in\N$, the above problem is well-posed and has a unique solution $u_n\in\HnsO$. Furthermore, since $\mA_n\in\MabR$, $n\in\N$, and since $f$ and $\mv$ are fixed, we have that $(\tilde{f}_n)_n$ is a bounded sequence in $\HmsO$. Hence, from the Poincar\'{e} inequality \eqref{poincare} and the a priori estimate \eqref{a-pr-est}, it follows that $(u_n)_n$ is bounded in $\HnsO$ and $(\mA_n(\grads u_n+\mv))_n$ is bounded in $\Ld{\Rd;\Rd}$. Let $u\in\HnsO$ and $\mq\in\Ld{\Rd;\Rd}$ be such that (by passing to subsequences) we obtain 
    $$
    u_n\xrightharpoonup{\HnsO}u \qquad \text{and} \qquad
 \mA_n(\grads u_n+\mv)\xrightharpoonup{\Ld{\Rd;\Rd}}\mq.
    $$
    By proving that $\mq=\mA(\grads u+\mv)$, a.e. on $\Rd$, we get the desired result. Having in mind Remark \ref{rem:intro-decomp} and Theorem \ref{thm:intro-Hs-MabR}, we shall consider separately the cases $x\in\Omega$ and $x\in(\cl\Omega)^{\text{c}}$. 
    
    First, let $\varphi\in{\rm C}^{\infty}_c((\cl\Omega)^{\text{c}};\Rd)$ and let $\Omega_1\subseteq\Rd$ be an open and bounded set such that $\Omega\Subset\Omega_1$ and $\supp\varphi\subseteq(\cl\Omega_1)^{\text{c}}$. 
    From the $H^s$-convergence of $\mA_n^{\rm T}$ to $\mA^{\rm T}$, by Theorem \ref{thm:intro-Hs-MabR} it follows that $\mA_n^{\rm T}\xrightharpoonup{*}\mA^{\rm T}$ on $(\cl\Omega)^{\text{c}}$. This, together with $u_n\xrightharpoonup{\HnsO}u$ and Theorem \ref{KS-Lema:2.12}, implies
    \begin{align*}
    \lim\limits_{n\to\infty}\int\limits_{(\cl\Omega)^{\text{c}}}\mA_n(\grads u_n+\mv)\cdot\varphi\,dx &= \lim\limits_{n\to\infty}\int\limits_{(\cl\Omega_1)^{\text{c}}}(\grads u_n+\mv)\cdot\mA_n^{\rm T}\varphi\,dx \\
    &= \int\limits_{(\cl\Omega_1)^{\text{c}}}(\grads u+\mv)\cdot\mA^{\rm T}\varphi\,dx\\
    &= \int\limits_{(\cl\Omega)^{\text{c}}}\mA(\grads u+\mv)\cdot\varphi\,dx \,. 
    \end{align*} 
    Since $\varphi\in{\rm C}^{\infty}_c((\cl\Omega)^{\text{c}};\Rd)$ was arbitrary, we conclude $\mq=\mA(\grads u+\mv)$ a.e. on $(\cl\Omega)^{c}$.

    It remains to prove $\mq=\mA(\grads u+\mv)$ a.e. on $\Omega$. Let $w_0\in\HnsO$ and define $g\in\HmsO$ by $g:=-\divs(\mA^{\rm T}\grads w_0)$. We take $(w_n)_n$ in $\HnsO$ to be the sequence of solutions to
    $$
    -\divs(\mA^{\rm T}_n\grads w_n) = g.
    $$
    Then, from the $H^s$-convergence of $\mA_n^{\rm T}$ to $\mA^{\rm T}$ and the choice of $g$, by Theorem \ref{thm:intro-Hs-MabR} we get that
    $$
    w_n\xrightharpoonup{\HnsO}w_0 \qquad \text{and} \qquad \mA_n^{\rm T}\grads w_n \xrightharpoonup{\Ld{\Rd;\Rd}} \mA^{\rm T}\grads w_0.
    $$
    Now, let $\varphi\in {\rm C}^{\infty}_c(\Omega)$. We have
    $$
    \int\limits_{\Rd}\varphi\mA_n(\grads u_n+\mv)\cdot\grads w_n\,dx = 
    \int\limits_{\Rd}\varphi(\grads u_n+\mv)\cdot \mA_n^{\rm T}\grads w_n\,dx,
    $$
    and by letting $n\to\infty$, and applying the fractional compensated compactness result, on the left-hand side we get
    $$
    \lim\limits_{n\to\infty}\int\limits_{\Rd}\varphi\mA_n(\grads u_n+\mv)\cdot\grads w_n\,dx=\int\limits_{\Rd}\varphi\,\mq \cdot\grads w_0\,dx,
    $$
    while on the right-hand side
    \begin{align*}
        \lim\limits_{n\to\infty}\int\limits_{\Rd}\varphi(\grads u_n+\mv)\cdot \mA_n^{\rm T}\grads w_n\,dx &=  \lim\limits_{n\to\infty}\int\limits_{\Rd}\varphi\grads u_n\cdot \mA_n^{\rm T}\grads w_n\,dx + \lim\limits_{n\to\infty}\int\limits_{\Rd}\varphi\,\mv\cdot \mA_n^{\rm T}\grads w_n\,dx\\
        &= \int\limits_{\Rd}\varphi\grads u\cdot \mA^{\rm T}\grads w_0\,dx + \int\limits_{\Rd}\varphi\,\mv\cdot \mA^{\rm T}\grads w_0\,dx\\
        &= \int\limits_{\Rd}\varphi\,\mA(\grads u+\mv)\cdot \grads w_0\,dx.
    \end{align*}
    Hence, for every $\varphi\in{\rm C}_c^{\infty}(\Omega)$, we get
    $$
    \int\limits_{\Rd}\varphi\,\mq \cdot\grads w_0\,dx = \int\limits_{\Rd}\varphi\mA(\grads u+\mv)\cdot \grads w_0\,dx,
    $$
    which implies 
    \begin{equation}\label{qw_0}
    (\mA(\grads u+\mv)\cdot \grads w_0)(x) = (\mq \cdot\grads w_0)(x), \quad \text{for a.e. } x\in\Omega.
    \end{equation}
    Since the set of all Lebesgue points of $(\mA(\grads u+\mv)-\mq)$, denoted by $\tilde{\Omega}$, is contained in the set of all Lebesgue points of $(\mA(\grads u+\mv)-\mq)\cdot\grads w_0$, for any $w_0\in {\rm C}_c^{\infty}(\Omega)$ (and since $w_0\in\HnsO$ was arbitrary), we get that \eqref{qw_0} holds in $\tilde{\Omega}$ for any $w_0\in {\rm C}_c^{\infty}(\Omega)$.
    By taking a suitable $w_0\in {\rm C}_c^{\infty}(\Omega)$ for each $x\in\tilde{\Omega}$ and applying Lemma \ref{lem:grads=lambda}, we conclude
    $$
    \mq (x) = \mA(\grads u+\mv)(x), \quad \text{for a.e. } x \in\Omega.
    $$
\end{proof}

From \v Zikov's lemma, by taking $\mv=0$, we immediately obtain transpose invariance of the $H^s$-limit. More precisely, the following is valid.

\begin{corollary}\label{cor:transpose_cnvg} Let $(\mA_n)_n$ in $\MabR$, and $\mA\in\MabR$. Then
    $\mA_n$  $H^s$-converges to $\mA$ if and only if $\mA_n^{\rm T}$  $H^s$-converges to  $\mA^{\rm T}$.
    
\end{corollary}

Moreover, as a consequence of the fractional compensated compactness result and due to the local nature of weak-$\ast$ convergence, the energy convergence property stated in Proposition \ref{prop:energy-conv}  and the locality property from Proposition \ref{prop:locality} remain to hold for sequences $(\mA_n)_n$ in $\MabR$.  
In addition, there exists a metric on $\MabR$ that generates $H^s$- topology on $\MabR$. 
To be precise, let $d_{\ast}$ be the weak-$\ast$ metric on $\lM(\alpha,\beta;\R^d\setminus \cl\Omega)$, and let $d_H$ be the $H$-metric on $\Mab$. Define $d:\MabR\to[0,\infty\rangle$ such that, for $\mA,\mB\in\MabR$,
\begin{equation}\label{eq:metricForHsRdArbitComplement}
d(\mA,\mB) = d_H(\mA|_{\Omega},\mB|_{\Omega}) + d_{\ast}(\mA|_{\R^d\setminus \cl\Omega},\mB|_{\R^d\setminus \cl\Omega}).
\end{equation}
Since both $d_H$ and $d_{\ast}$ are metrics on $\Mab$ and $\lM(\alpha,\beta;\R^d\setminus \cl\Omega)$, respectively, we get that $d$ is a well-defined metric on $\MabR$. Furthermore, by Proposition \ref{prop:Hs_MabR_unique} and Theorem \ref{thm:intro-Hs-MabR}, we get
$$
\mA_n\ H^s\text{-converges  to } \mA_{\infty} \quad \iff \quad  d(\mA_n,\mA_{\infty}) \to 0.
$$

\section{Schur topology}\label{sec:schur}
In this section only, the generic codomain of function spaces will be $\mathbb{C}$. 

Given local coefficients $\mA\in {\rm L}^\infty(U)^{d\times d}$ for a measurable $U\subseteq \R^d$, we can also consider them as a bounded operator $\mA\in L({\rm L}^2(U)^d)$ by mapping $v\in {\rm L}^2(U)^d$
to $\mA(\cdot)v(\cdot)$. Applying, e.g., the Helmholtz decomposition, one can show that ${\rm L}^\infty(U)^{d\times d}\subsetneq L({\rm L}^2(U)^d)$, and that, in general, the local homogenisation approaches fail if possibly nonlocal coefficients, i.e., elements of $L({\rm L}^2(U)^d)$, are considered (cf., e.g.,~\cite{Wau18} for details). In \cite{Wau18}, the following approach to homogenisation of such nonlocal coefficients was introduced and later refined in \cite{Wau25}.

Let $\mathcal{H}$ be a separable Hilbert space (linear in the second argument), $\mathcal{H}_0$ a closed subspace, and $\mathcal{H}_1\coloneqq \mathcal{H}_0^\perp$, i.e., $\mathcal{H}=\mathcal{H}_0\oplus\mathcal{H}_1$. Then, we can identify a bounded operator $a\in L(\mathcal{H})$ as $a=\begin{psmallmatrix}
    a_{00}&a_{01}\\
    a_{10}&a_{11}
\end{psmallmatrix}$, where $a_{ij}\coloneqq \pi_i a\iota_j$ for $i,j\in\{0,1\}$. Here, $\pi_i\colon\mathcal{H}\to\mathcal{H}_i$ denotes the orthogonal projection onto $\mathcal{H}_i$, and $\iota_i\colon\mathcal{H}_i\to\mathcal{H}$ the canonical embedding.
We endow the set
\begin{equation*}
    \mathfrak{M}(\mathcal{H}_0,\mathcal{H}_1)\coloneqq\bigl\{a\in L(\mathcal{H}): a_{00}^{-1}\in L(\mathcal{H}_0)\text{ and } a^{-1}\in L(\mathcal{H})\bigr\}
\end{equation*}
with the inital topology $\tau(\mathcal{H}_0,\mathcal{H}_1)$ with respect to
\begin{align*}
    \Psi_{00}&\colon\mathfrak{M}(\mathcal{H}_0,\mathcal{H}_1)\to L(\mathcal{H}_0);& a&\mapsto a_{00}^{-1}\text{,}\\
        \Psi_{10}&\colon\mathfrak{M}(\mathcal{H}_0,\mathcal{H}_1)\to L(\mathcal{H}_0,\mathcal{H}_1);& a&\mapsto a_{10}a_{00}^{-1}\text{,}\\
            \Psi_{01}&\colon\mathfrak{M}(\mathcal{H}_0,\mathcal{H}_1)\to L(\mathcal{H}_1,\mathcal{H}_0);& a&\mapsto a_{00}^{-1}a_{01}\text{, and}\\
                \Psi_{11}&\colon\mathfrak{M}(\mathcal{H}_0,\mathcal{H}_1)\to L(\mathcal{H}_1);& a&\mapsto a_{11}- a_{10}a_{00}^{-1}a_{01}\text{,}
\end{align*}
where we consider the codomains equipped with the respective weak operator
topology. 
We refer to $\tau(\mathcal{H}_0,\mathcal{H}_1)$ as the \textbf{Schur topology}, and we call convergence with respect to this topology \textbf{nonlocal $H$-convergence}.

Moreover, for $\gamma\coloneqq \begin{psmallmatrix}
    \gamma_{00}&\gamma_{01}\\
    \gamma_{10}&\gamma_{11}
\end{psmallmatrix}\in \R^{2\times 2}_{>0}\text{,}$
we define\footnote{For a Hilbert space $\mathcal{K}$, a bounded linear operator $B\in L(\mathcal{K})$, and $c>0$, we write $\Re B\geq c$ as a short form of
$\Re\langle v,Bv \rangle_{\mathcal{K}}\geq c\lVert v\rVert^2_{\mathcal{K}}$, $ v\in\mathcal{K}$.}
\begin{multline*}
    \mathfrak{M}(\gamma,\mathcal{H}_0,\mathcal{H}_1)\coloneqq\bigl\{a\in \mathfrak{M}(\mathcal{H}_0,\mathcal{H}_1): \Re\Psi_{00}(a)^{-1}\geq \gamma_{00},\; \Re\Psi_{00}(a)\geq 1/\gamma_{11},\\
    \lVert\Psi_{10}(a)\rVert\leq \gamma_{10}, \;\lVert\Psi_{01}(a)\rVert\leq \gamma_{01},\; \Re\Psi_{11}(a)^{-1}\geq 1/\gamma_{11},\;\Re\Psi_{11}(a)\geq \gamma_{00}\bigr\}\text{,}
\end{multline*}
and we endow this set with the trace (or relative) topology of $\tau(\mathcal{H}_0,\mathcal{H}_1)$ (which we will still denote the same way). For these spaces, we obtain the following compactness and metrisability theorem (note that the statement in \cite{Wau18} appears to be less general on the first glance. The proof of the general case, however, can be read off right away without difficulty).

\begin{theorem}[{\cite[Chp.~5]{Wau18}}]\label{thm:CompactAndMetricSchurTopo}
    Let $\gamma\in \R^{2\times 2}_{>0}$. Then,
    $(\mathfrak{M}(\gamma,\mathcal{H}_0,\mathcal{H}_1), \tau(\mathcal{H}_0,\mathcal{H}_1))$ is a compact and metrisable space.
\end{theorem}

In order to compare this result with Theorem~\ref{thmLocalHisMetrAndComp}, we need to provide a specific decomposition of the corresponding $\mathcal{H}={\rm L}^2(\Omega)^d$. Writing $\nabla_0$ for $\nabla$ on ${\rm H}_0^1(\Omega)\subseteq {\rm L}^2(\Omega)$, we infer, by the
classical Poincar\' e inequality, that $\ran(\nabla_0)$ is a closed subset of ${\rm L}^2(\Omega)^d$. Nonlocal $H$-convergence is a generalisation of $H$-convergence from ${\rm L}^\infty(\Omega;\R^{d\times d})$ to $L({\rm L}^2(\Omega)^d)$, i.e., from local to nonlocal coefficients, in the following sense.

\begin{theorem}\label{thm:SchurTopoCoincHTopoLocCoeff}
    Let $0<\alpha\leq \beta$. Then,
    $$
    (\Mab,\tau_H)=(\Mab, \tau(\ran(\nabla_0),\ran(\nabla_0)^\perp))\text{,}
    $$
    where $\tau(\ran(\nabla_0),\ran(\nabla_0)^\perp)$ stands for the corresponding trace topology and $\tau_H$ is from
    Theorem~\ref{thmLocalHisMetrAndComp}.
\end{theorem}

For a proof, based on the classical Helmholtz decomposition for $d=3$ and a sufficiently regular $\Omega$, see~\cite[Sec.~1.6]{B25}. For our more general assumptions on $d$ and $\Omega$, the proof works analogously to the following considerations.

Let $\grads_\Omega$ denote $\grads$ on $\HnsO\subseteq {\rm L}^2(\Omega)$. Then, by the fractional Poincar\' e inequality (see Theorem~\ref{thm:properties_d=1}), $\ran(\grads_\Omega)$ is a closed subset of ${\rm L}^2(\mathbb{R}^d)^d$.
By~\cite[Thm.~2.9]{Wau18} (see~\cite{TW14} for details), for $\gamma\in \R^{2\times 2}_{>0}$, $a\in\mathfrak{M}(\gamma,\ran(\grads_\Omega),\ran(\grads_\Omega)^\perp)$, and $f\in {\rm H}^{-s}(\Omega)$, 
\begin{equation}\label{eq:SchurFracGradPDEasInnProd}
(\forall v\in \HnsO) \quad \langle a\grads_\Omega u, \grads_\Omega v\rangle_{{\rm L}^2(\mathbb{R}^d)^d}=f(v)
\end{equation}
is well-posed with a unique solution $u\in \HnsO$.
Furthermore, from the invertibility of the Schur complement of $a$, i.e.,
    $(a^{-1})_{11}^{-1}=a_{11}-a_{10}a^{-1}_{00}a_{01}$, we readily deduce that, for $z\in \ran(\grads_\Omega)^\perp\subseteq {\rm L}^2(\mathbb{R}^d)^d$,
\begin{equation}\label{eq:dualSchurFracGradPDEasInnProd}
\bigl(\forall q\in \ran(\grads_\Omega)^\perp\bigr) \quad \langle a^{-1}p, q\rangle_{{\rm L}^2(\mathbb{R}^d)^d}=\langle z, q\rangle_{{\rm L}^2(\mathbb{R}^d)^d}
\end{equation}
has the unique solution $p=\iota_1(a_{11}-a_{10}a^{-1}_{00}a_{01})z\in\ran(\grads_\Omega)^\perp\subseteq {\rm L}^2(\mathbb{R}^d)^d$.

We can also regard $\tau(\ran(\grads_\Omega),\ran(\grads_\Omega)^\perp)$ for $\mathcal{H}={\rm L}^2(\mathbb{R}^d)^d$, and, in combination with~\eqref{eq:SchurFracGradPDEasInnProd} and~\eqref{eq:dualSchurFracGradPDEasInnProd},
we obtain the following characterisation of convergence with respect to $\tau(\ran(\grads_\Omega),\ran(\grads_\Omega)^\perp)$.
\begin{lemma}\label{lemma:SchurConvEquivLocalPlusDual}
    Let $\gamma\in \R^{2\times 2}_{>0}$. Then, a sequence $(a_n)_n$ from $\mathfrak{M}(\gamma,\ran(\grads_\Omega),\ran(\grads_\Omega)^\perp)$ converges to $a\in\mathfrak{M}(\gamma,\ran(\grads_\Omega),\ran(\grads_\Omega)^\perp)$
    with respect to $\tau(\ran(\grads_\Omega),\ran(\grads_\Omega)^\perp)$ if and only if all of the following conditions hold
    \begin{enumerate}
        \item[(a)] For each $f\in {\rm H}^{-s}(\Omega)$, the sequence $(u_n)_n$ in $\HnsO$ of unique solutions to~\eqref{eq:SchurFracGradPDEasInnProd}, where $a_n$ replaces $a$, weakly converges in $\HnsO$ to the unique solution $u\in\HnsO$ to~\eqref{eq:SchurFracGradPDEasInnProd}, 
        \item[(b)] and $a_n \grads_\Omega u_n\xrightharpoonup{{\rm L}^2(\R^d)^d} a \grads_\Omega u$.
        \item[(c)] For each $z\in \ran(\grads_\Omega)^\perp\subseteq {\rm L}^2(\mathbb{R}^d)^d$, the sequence $(p_n)_n$ in $\ran(\grads_\Omega)^\perp\subseteq {\rm L}^2(\mathbb{R}^d)^d$ of unique solutions to~\eqref{eq:dualSchurFracGradPDEasInnProd}, where $a^{-1}_n$ replaces $a^{-1}$, weakly converges in ${\rm L}^2(\mathbb{R}^d)^d$ to the unique solution $p\in\ran(\grads_\Omega)^\perp\subseteq {\rm L}^2(\mathbb{R}^d)^d$ to~\eqref{eq:dualSchurFracGradPDEasInnProd},
        \item[(d)] and $a_n^{-1} p_n\xrightharpoonup{{\rm L}^2(\R^d)^d} a^{-1} p$.
    \end{enumerate}
\end{lemma}
\begin{proof}
This proof works analogously to the proof of~\cite[Thm.~4.1]{Wau18}.

    The equivalence of (a) and (b) on the one hand and $\Psi_{00}(a_n)\to \Psi_{00}(a)$ and $\Psi_{10}(a_n)\to \Psi_{10}(a)$ in the weak operator topology on the other hand follow exactly like~\cite[Thm.~4.10~(a) and~(b)]{Wau18}.

    Next, due to the solution formula for~\eqref{eq:dualSchurFracGradPDEasInnProd}, (c) and $\Psi_{11}(a_n)\to \Psi_{11}(a)$ in the weak operator topology are clearly equivalent.

    Finally, a tedious but straightforward calculation shows $\iota_0a^{-1}_{00}a_{01}z=\iota_1 z-a^{-1}p$
    which immediately yields the equivalence of (d) and
    $\Psi_{01}(a_n)\to \Psi_{01}(a)$ in the weak operator topology.
    
\end{proof}

Combining Lemma~\ref{lemma:SchurConvEquivLocalPlusDual}
and Proposition~\ref{prop:ZikovLemmaFrac}, and imitating the proof of~\cite[Thm.~1.6.7]{B25}, we infer the following fractional version of
Theorem~\ref{thm:SchurTopoCoincHTopoLocCoeff}.

\begin{theorem}\label{thm:NonlocalTopologieFracGradPlusOrtho}
    Let $0<\alpha\leq \beta$. Then,
    $$
    (\MabR,\tau_d)=(\MabR, \tau(\ran(\grads_\Omega),\ran(\grads_\Omega)^\perp))\text{,}
    $$
    where $\tau(\ran(\grads_\Omega),\ran(\grads_\Omega)^\perp)$ stands for the corresponding trace topology and $\tau_d$ is the
    topology arising from the metric~\eqref{eq:metricForHsRdArbitComplement}.
\end{theorem}
\begin{proof}
    First, it is easy to check that $\MabR$ is a subset of $\mathfrak{M}(\gamma,\ran(\grads_\Omega),\ran(\grads_\Omega)^\perp)$, where $\gamma\coloneqq \begin{psmallmatrix}
        \alpha &\beta/\alpha\\
        \beta/\alpha &\beta
    \end{psmallmatrix}$.
    
    By Lemma~\ref{lemma:SchurConvEquivLocalPlusDual}, a sequence that converges with respect to $\tau(\ran(\grads_\Omega),\ran(\grads_\Omega)^\perp)$ also converges
    with respect to $\tau_d$, i.e., the identity mapping is sequentially continuous. Because of the compactness of
    \begin{equation}\label{eq:ViableSchurSupSpaceForLocalFracH}
    (\mathfrak{M}(\gamma,\ran(\grads_\Omega),\ran(\grads_\Omega)^\perp),\tau(\ran(\grads_\Omega),\ran(\grads_\Omega)^\perp))
    \end{equation}
    (see Theorem~\ref{thm:CompactAndMetricSchurTopo}) and the metrisability of both discussed topologies, it suffices to show that $\MabR$ is a closed subset of~\eqref{eq:ViableSchurSupSpaceForLocalFracH} in order to conclude the statement.

    Let $(\mA_n)_n$ be a sequence in $\MabR$ that converges to $a\in \mathfrak{M}(\gamma,\ran(\grads_\Omega),\ran(\grads_\Omega)^\perp)$ with respect to $\tau(\ran(\grads_\Omega),\ran(\grads_\Omega)^\perp)$. By Corollary~\ref{compactness_result}, for every subsequence of $(\mA_n)_n$, there exist another subsequence (which we do not relabel) $(\mA_n)_n$ and $\mA\in \MabR$ such that $(\mA_n)_n$ $H^s$-converges to $\mA$. Let $z\in \ran(\grads_\Omega)^\perp\subseteq {\rm L}^2(\mathbb{R}^d)^d$, and consider the sequence $(p_n)_n$ in $\ran(\grads_\Omega)^\perp\subseteq {\rm L}^2(\mathbb{R}^d)^d$ of unique solutions to~\eqref{eq:dualSchurFracGradPDEasInnProd}, where $\mA^{-1}_n$ replaces $a^{-1}$, as well as the unique solution $p\in\ran(\grads_\Omega)^\perp\subseteq {\rm L}^2(\mathbb{R}^d)^d$ to~\eqref{eq:dualSchurFracGradPDEasInnProd}, where $\mA^{-1}$ replaces $a^{-1}$. This implies 
    $$
    \mA_n^{-1}p_n-z,\mA^{-1}p-z\in\ran(\grads_\Omega)\subseteq {\rm L}^2(\mathbb{R}^d)^d.
    $$
    Hence, there are $u_n,u\in \HnsO$ such that
    $$
    p_n=\mA_n(\grads_\Omega u_n + z)\text{ and }p=\mA(\grads_\Omega u + z),
    $$
    and consequently
    $$
    (\forall v\in \HnsO) \quad \langle \mA_n(\grads_\Omega u_n+z), \grads_\Omega v\rangle_{{\rm L}^2(\mathbb{R}^d)^d}=\langle \mA(\grads_\Omega u+z), \grads_\Omega v\rangle_{{\rm L}^2(\mathbb{R}^d)^d}.
    $$
    Therefore, Proposition~\ref{prop:ZikovLemmaFrac} and
    Corollary~\ref{cor:transpose_cnvg} yield
    $$
    \mA_n^{-1} p_n=\grads_\Omega u_n + z\xrightharpoonup{{\rm L}^2(\R^d)^d} \grads_\Omega u + z=\mA^{-1} p,
    $$
    as well as
    $$
    p_n=\mA_n(\grads_\Omega u_n + z)\xrightharpoonup{{\rm L}^2(\R^d)^d} \mA(\grads_\Omega u + z)=p,
    $$
    and Lemma~\ref{lemma:SchurConvEquivLocalPlusDual}
    shows $a=\mA$.
\end{proof}
\begin{remark}
    Comparing Theorem~\ref{thm:NonlocalTopologieFracGradPlusOrtho}, Theorem~\ref{thm:SchurTopoCoincHTopoLocCoeff}, and the metric~\eqref{eq:metricForHsRdArbitComplement}, we can characterise $H^s$-convergence on $\MabR$ via two natural Schur topologies. On the one hand, we have $\tau(\ran(\grads_\Omega),\ran(\grads_\Omega)^\perp)$ due to Theorem~\ref{thm:NonlocalTopologieFracGradPlusOrtho}. On the other hand, consider the decomposition $\Ld{\Rd}^d=\Ld{\Omega}^d \oplus \Ld{\Rd\setminus\cl\Omega}^d$.
    Note that the Schur topology for $\mathcal{H}=\{0\}\oplus\mathcal{H}$ exactly is the weak operator topology on $L(\mathcal{H})$. Furthermore, by compactness and factorisation of ${\rm L}^1$-functions into the product of two ${\rm L}^2$-functions, weak-$\ast$ limits on $\lM(\alpha,\beta;U)\subseteq {\rm L}^\infty(U)^{d\times d}$ exactly coincide with weak operator topology limits on $\lM(\alpha,\beta;U) \subseteq L({\rm L}^2(U)^d)$ for a measurable $U\subseteq \R^d$.
    Thus, Theorem~\ref{thm:SchurTopoCoincHTopoLocCoeff} implies that $\tau(\ran(\nabla_0)\oplus\{0\},\ran(\nabla_0)^\perp\oplus\Ld{\Rd\setminus\cl\Omega}^d)$ exactly characterises convergence with respect to the metric~\eqref{eq:metricForHsRdArbitComplement}.
    In other words,
    \begin{multline*}  (\MabR,\tau(\ran(\grads_\Omega),\ran(\grads_\Omega)^\perp))\\
    =(\MabR,\tau(\ran(\nabla_0)\oplus\{0\},\ran(\nabla_0)^\perp\oplus\Ld{\Rd\setminus\cl\Omega}^d)),
    \end{multline*}
    i.e., the corresponding trace topologies coincide.

    How differently these two topologies behave for nonlocal coefficients is an open question.
\end{remark}
We close this section with an application of the Schur topology that yields a homogenisation result for a heat type equation with local coefficients and nonlocal differential operators. Before, however, we come to the corresponding proof, we briefly provide a uniqueness statement for solutions of the fractional heat equation provided there. For this, we put 
\[
\divs_{\Omega,w}\coloneqq -(\grads_\Omega)^{\ast},
\]
i.e., as the ${\rm L}^2$-adjoint of the fractional gradient on $\HnsO$. Then, by \cite[Proposition 9.2.2(b)]{SeTrWa22}, $\divs_{\Omega,w}\subseteq \divs_{\Omega}$ (see \eqref{eq:divOm}).
Furthermore, we can consider its domain ${\rm H}(\divs_{\Omega,w})$ as a Hilbert space endowed with the graph inner product.

For $f\in  {\rm L}^2(0,T;{\rm L}^2(\Omega))$ we call $u$ a \textbf{solution} of the equation 
\[
   \partial_t u -\divs_{\Omega,w}\mA \grads u = f \text{ on }\langle 0,T\rangle\times \Omega;\quad u(0)=0,
\]   
if $u\in {\rm H}^1(0,T;{\rm L}^2(\Omega))$ with $u(0)=0$ and, $u \in {\rm L}^2(0,T;\dom(\divs_{\Omega,w} \mA \grads))$ such that for a.e.~$t\in \langle 0,T\rangle$ we have
\[
   u'(t) - \divs_{\Omega,w} \mA \grads u(t)=f(t),
\]where $\partial_t u(t)=u'(t)$ is the distributional derivative of $u$ with respect to $t$.

\begin{lemma}\label{lem:heatunique} Let $\mA \in \MabR$, $\Omega\subseteq \R^d$ open, $T>0$, $f\in {\rm L}^2(0,T;{\rm L}^2(\Omega))$. Then, there exists at most one $u \in {\rm H}^1(0,T;{\rm L}^2(\Omega))\cap {\rm L}^2(0,T;\dom(\divs_{\Omega,w} \mA \grads))$ such that
\[
   \partial_t u -\divs_{\Omega,w}\mA \grads u = f,\quad u(0)=0.
\]    
\end{lemma}
\begin{proof} Note that by the (vector-valued) Sobolev embedding theorem, the pointwise evaluation of $u$ at every fixed time $t$ taking values in ${\rm L}_2(\Omega)$ is well-defined. By linearity, we may assume without restriction that $f=0$. Then, let $0<t<T$. We apply $\langle \cdot, u\rangle_{{\rm L}^2(0,t;{\rm L}^2(\Omega))}$ to the equation satisfied by $u$ and obtain
\begin{align*}
  0&= \Re \langle \partial_t u -\divs_{\Omega,w}\mA \grads u,u\rangle_{{\rm L}^2(0,t;{\rm L}^2(\Omega))}  \\
   & = \Re \big(\langle \partial_t u, u\rangle_{{\rm L}^2(0,t;{\rm L}^2(\Omega))} + \langle \mA \grads u,\grads u\rangle_{{\rm L}^2(0,t;{\rm L}^2(\Omega))}\big) \\
   & \geq \Re \langle \partial_t u, u\rangle_{{\rm L}^2(0,t;{\rm L}^2(\Omega))} \\
   & = \frac{1}{2} \|u(t)\|^2_{{\rm L}^2(\Omega)}-\frac{1}{2}\|u(0)\|^2_{{\rm L}^2(\Omega)} =\frac{1}{2} \|u(t)\|^2_{{\rm L}^2(\Omega)},
\end{align*}
where we used integration by parts in the second to last step and the initial condition in the last step. We infer $u(t)=0$ for all $t\in [0,T\rangle$ and, thus, obtain $u=0$.
\end{proof}

\paragraph{Proof of Theorem \ref{thm:intro-evolution-Hs}.}
In order to prove this theorem, we will employ the theory of evolutionary equations. For a detailed introduction that covers all the concepts used in this proof, see~\cite{SeTrWa22}. Note that uniqueness of solutions of the fractional heat equation has been addressed already in Lemma \ref{lem:heatunique}. Thus, it suffices to employ the existence theorem in the framework of evolutionary equations in order to obtain well-posedness of the equations at hand.

We introduce ${\rm L}_{\nu}^2(\R; \mathcal{K})$ for a fixed $\nu>0$ and any Hilbert space $\mathcal{K}$ as the weighted ${\rm L}^2$-space
$$
\big\{f \mid f\colon\R\to \mathcal{K} \text{ Bochner-meas.\ and } (t\mapsto \exp(-\nu t)\lVert f(t)\rVert_{\mathcal{K}})\in{\rm L}^2 (\R)\big\}\text{.}
$$
On these spaces, we denote by $\partial_{t,\nu}$ the weak derivative with maximal domain.
We write ${\rm H}^1_{\nu}(\R;\mathcal{K})$ if we consider these domains as Hilbert spaces endowed with the graph inner product.
Extending $f$ to ${\rm L}_{\nu}^2(\R; {\rm L}^2(\Omega))$ by $0$, we obtain unique
$$
(\tilde u,\tilde v),(\tilde u_n,\tilde v_n)\in \big({\rm H}^1_{\nu}(\R;{\rm L}^2(\Omega))\cap {\rm L}_{\nu}^2(\R;\HnsO )\big) \times {\rm L}_{\nu}^2(\R; {\rm H}(\divs_{\Omega,w}))
$$
with
$$
\Big(\partial_{t,\nu}\begin{pmatrix}
    1&0\\
    0&0
\end{pmatrix}+\begin{pmatrix}
    0&0\\
    0&\mA_n^{-1}
\end{pmatrix}+\begin{pmatrix}
    0&\divs_{\Omega,w}\\
    \grads_\Omega&0
\end{pmatrix}\Big)\begin{pmatrix}
    \tilde u_n\\\tilde v_n
\end{pmatrix}=\begin{pmatrix}
    f\\0
\end{pmatrix}\text{,}
$$
and
$$
\Big(\partial_{t,\nu}\begin{pmatrix}
    1&0\\
    0&0
\end{pmatrix}+\begin{pmatrix}
    0&0\\
    0&\mA^{-1}
\end{pmatrix}+\begin{pmatrix}
    0&\divs_{\Omega,w}\\
    \grads_\Omega&0
\end{pmatrix}\Big)\begin{pmatrix}
    \tilde u\\\tilde v
\end{pmatrix}=\begin{pmatrix}
    f\\0
\end{pmatrix}\text{,}
$$
i.e., well-posedness in the sense of Picard (see~\cite[Thm.~6.2.1]{SeTrWa22})
and maximal regularity (see~\cite[Thm.~15.2.3]{SeTrWa22}).
In particular, by the Sobolev embedding theorem and causality, $\tilde u_n$ and $\tilde u$ are continuous with $\tilde u_n(t)=\tilde u(t)=0$ for $t\leq 0$.
Moreover, by maximal regularity, there also exists a $\tilde C>0$ with the a priori estimate
$$
\lVert\tilde u_n\rVert_{{\rm H}^1_{\nu}(\R;{\rm L}^2(\Omega))}
+\lVert\tilde u_n\rVert_{{\rm L}_{\nu}^2(\R;\HnsO )}\leq \tilde C\lVert f \rVert_{{\rm L}_{\nu}^2}
\text{ and } \lVert\tilde u\rVert_{{\rm H}^1_{\nu}(\R;{\rm L}^2(\Omega))}
+\lVert\tilde u\rVert_{{\rm L}_{\nu}^2(\R;\HnsO )}\leq \tilde C\lVert f \rVert_{{\rm L}_{\nu}^2}.
$$
By virtue of Theorem~\ref{thm:NonlocalTopologieFracGradPlusOrtho}, and the fractional Poincar\' e inequality (see Theorem~\ref{thm:properties_d=1}), we can also apply~\cite[Thm.~4.1]{BFSW25} and~\cite[Thm.~6.5]{BSW24} in order to obtain $\tilde u_n \rightharpoonup \tilde u$ weakly in ${\rm L}_{\nu}^2(\R;{\rm L}^2(\Omega))$.

Finally, we define $u_n\in {\rm L}^2(0,T;{\rm L}^2(\Omega))$ and $u\in {\rm L}^2(0,T;{\rm L}^2(\Omega))$ to be the restrictions to $[0,T\rangle$ of $\tilde u_n$ and $\tilde u$ respectively. It readily follows that
\begin{align*}
u_n &\in {\rm H}^1(0,T;{\rm L}^2(\Omega))\cap {\rm L}^2(0,T; \dom (\divs_{\Omega,w}\mA_n\grads_\Omega)) \text{ and}\\
u &\in {\rm H}^1(0,T;{\rm L}^2(\Omega))\cap {\rm L}^2(0,T; \dom (\divs_{\Omega,w}\mA\grads_\Omega)),
\end{align*}
that these are solutions to~\eqref{eq:ExampleFracHeatEqAn} and~\eqref{eq:ExampleFracHeatEqLimit} respectively, and that there exists a $C>0$ with
\begin{equation}\label{eq:MaxRegH1Bounds0T}
\begin{gathered}
\lVert u_n\rVert_{{\rm H}^1(0,T;{\rm L}^2(\Omega))}
+\lVert u_n\rVert_{{\rm L}^2(0,T;\HnsO )}\leq C\lVert f \rVert_{{\rm L}_{\nu}^2}
\text{ and}\\ \lVert u\rVert_{{\rm H}^1(0,T;{\rm L}^2(\Omega))}
+\lVert u\rVert_{{\rm L}^2(0,T;\HnsO )}\leq C\lVert f \rVert_{{\rm L}_{\nu}^2}.
\end{gathered}
\end{equation}
Additionally, we infer $u_n \rightharpoonup u$ weakly in ${\rm L}^2(0,T;{\rm L}^2(\Omega))$.
Together with~\eqref{eq:MaxRegH1Bounds0T} and the fractional Rellich--Kondrašov theorem (see Theorem~\ref{thm:properties_d=1}),
the Aubin--Lions lemma implies
$u_n \to u$ in ${\rm L}^2(0,T;{\rm L}^2(\Omega))$.
    $\hfill\blacksquare$

\section{$G^s$-convergence}\label{sec:Gconv}
The concept of $G$-convergence is well established in the literature (see \cite[Section 1.3.2]{Allaire} and \cite[Chapter 6]{Tartar}). It can be regarded as a particular case of $H$-convergence, specifically for sequences of operators with symmetric coefficients. 
In the following, we introduce the concept of $G^s$-convergence and investigate its relationship with $H^s$-convergence, thereby extending the classical framework to a broader class of operators.

For $U\subseteq\Rd$ open, let us denote by $\lMsymU$ the subset of all $\mA\in\MabU$ such that $\mA=\mA^{\rm T}$ a.e.~in $U.$

\begin{definition}[$G^s$-convergence]\label{def:G-s-conv}
Let $(\mA_n)_n$ be a sequence in $\lMsymR$ and $\mA\in\lMsymR$.
Then $(\mA_n)_n$ is said to \textbf{$G^s$-converge} to $\mA$ ($\mA_n\xrightarrow{G^s}\mA$), if for all $f\in\HmsO$
the sequence $(u_n)_n$ of solutions of 
\begin{equation}
\left\{
	\begin{array}{rl}
		-\divs (\mA_n\grads u_n)=f \quad   &\rm{in}\,\, \Omega, \\
		u_n=0 \quad &{\rm in}\,\, \Rd\setminus\Omega \\
	\end{array}
	\right.
\end{equation}    
converges weakly in $\HnsO$ to the solution $u$ of the equation
\begin{equation}\label{fpp}
\left\{
	\begin{array}{rl}
		-\divs (\mA\grads u)=f \quad   &\rm{in}\,\, \Omega, \\
		u=0 \quad &{\rm in}\,\, \Rd\setminus\Omega. \\
	\end{array}
	\right.
\end{equation} 
\end{definition}

We shall prove that $G^s$- and $H^s$-convergence are equivalent on $\lMsymR$. To this end, we require the following additional result.
\begin{lemma}\label{G-conv-aux}
    Let $\mA, \mB\in \Lb{\Rd;\R^{d\times d}}$ be symmetric matrix functions  and $U \subseteq \Rd$ be an open set. Assume that for all $\xi\in\Rd$
    $$
    \mA(x)\xi\cdot\xi\ge 0 \quad \text{and}\quad \mB(x)\xi\cdot\xi\ge 0, \quad\text{ for a.e. }x\in U
    $$
    and 
    \begin{equation}\label{lema4.2}
        \int_\Rd\mA(x)\grads\psi(x)\cdot\grads\psi(x)\, dx=\int_\Rd\mB(x)\grads\psi(x)\cdot\grads\psi(x)\, dx, \quad \psi\in {\rm C}_c^\infty(U) \,.
    \end{equation}
    Then 
    $$
    \mA(x)=\mB(x), \text{ for a.e. } x\in \Rd.
    $$
\end{lemma}
\begin{proof}
The identity $\mA(x)=\mB(x)$ for a.e.~$x\in U$ was established in \cite[Lemma 4.2]{CCM}. To obtain the corresponding equality on $U^\text{c}$, we employ the same methodological approach as in the aforementioned result. More precisely, it is enough to notice that for every 
 $\varphi\in {\rm C}_c^\infty(\Rd)$ and for every $x_0\in U^{\textrm{c}}$ there exist $r_0\in\langle 0,\infty\rangle$ and $x_1\in\Rd$ such that for every $r\in \langle r_0,\infty\rangle$ we have
    $$
    \varphi_{x_0,x_1,r}(x):=\varphi(r(x-x_0)-x_1)\in {\rm C}_c^\infty (U). 
    $$
    Moreover, from $s$-homogeneity and translational invariance of the fractional gradient, we obtain
    $$
    \grads \varphi_{x_0,x_1,r}(x)=r^s\grads\varphi(r(x-x_0)-x_1).
    $$
The statement now follows by applying the same steps as in the proof of \cite[Lemma 4.2]{CCM}.\end{proof}
\begin{proposition}[The equivalence between $H^s$- and $G^s$-convergence]
\label{prop:Gs-Hs} 
Let $(\mA_n)_n $ be a sequence in $ \lMsymR$ and  $\mA \in  \MabR$. Then $(\mA_n)_n$ $H^s$-converges to $\mA$ if and only if $\mA$ is symmetric and $(\mA_n)_n$ $G^s$-converges to $\mA$.
\end{proposition}
\begin{proof}
    If $(\mA_n)_n$ $H^s$-converges to $\mA$, then by Proposition \ref{prop:Hs_MabR_unique} and Corollary \ref{cor:transpose_cnvg}, $\mA$ is symmetric and, therefore, $(\mA_n)_n$ also $G^s$-converges to $\mA$. 

On the other hand, assume that $\mA_n$ $G^s$-converges to $\mA\in\lMsymR$. Then, for every $f\in\HmsO$, the sequence $(u_n)_n$  of solutions of 
    \begin{equation}\label{G-eq}
\left\{
	\begin{array}{rl}
		-\divs (\mA_n\grads u_n)=f \quad   &\rm{in}\,\, \Omega, \\
		u_n=0 \quad &{\rm in}\,\, \Rd\setminus\Omega \\
	\end{array}
	\right.
\end{equation}    
converges weakly in $\HnsO$ to the solution $u$ of the equation
\begin{equation}\label{G:eq_hom}
\left\{
	\begin{array}{rl}
		-\divs (\mA\grads u)=f \quad   &\rm{in}\,\, \Omega, \\
		u=0 \quad &{\rm in}\,\, \Rd\setminus\Omega. \\
	\end{array}
	\right.
\end{equation} 
According to Corollary \ref{compactness_result}, there exists a subsequence of $(\mA_n)_n$ that $H^s$-converges to a limit $\mB$, which is symmetric by virtue of Corollary \ref{cor:transpose_cnvg}. Consequently, the limit function $u$ of the corresponding subsequence of solutions $(u_n)_n$ of \eqref{G-eq} is the unique solution in $H^s_0(\Omega)$ of
\begin{equation}\label{G:eq_hom2}
\left\{
	\begin{array}{rl}
		-\divs (\mB\grads u)=f \quad   &\rm{in}\,\, \Omega, \\
		u=0 \quad &{\rm in}\,\, \Rd\setminus\Omega. \\
	\end{array}
	\right.
\end{equation} 
Subtracting the two homogenised equations (\ref{G:eq_hom}) and (\ref{G:eq_hom2}), we obtain
$$
-\divs(\mA-\mB)\grads u=0\quad \text{in }\Omega,$$
which in a weak formulation reads
$$
\int_\Rd(\mA-\mB)\grads u\cdot\grads v=0, \quad v\in\HnsO.
$$
By the arbitrariness of $f$, this identity also holds for any $u\in\HnsO$. Thus, using Lemma \ref{G-conv-aux} it follows that $\mA=\mB$ a.e.~in $\Rd$, which concludes the proof.
\end{proof}

\begin{remark}
Combining Proposition \ref{prop:Gs-Hs} and Theorem \ref{thm:intro-Hs-MabR} with the well-known relationship between $H$- and $G$-convergence \cite[Proposition 1.3.11]{Allaire}, we obtain a characterisation of $G^s$-convergence in terms of classical local $G$-convergence and weak-$\ast$ convergence:
\begin{equation*}
 \mA_n   \stackrel{G^s}{\longrightarrow}  \mA
\iff  \mA_n|_\Omega \stackrel{G}{\longrightarrow} \mA|_\Omega \ \text{ and } \
\mA_n|_{\R^d\setminus \cl\Omega}
\xrightharpoonup{\ * \ }
\mA|_{\R^d\setminus \cl\Omega} \,.
\end{equation*}
\end{remark}

\section{Conclusion}

In this note, we discussed homogenisation problems in the context of fractional diffusion with general non-periodic highly oscillatory conductivities in divergence form. The spatial derivatives are provided by bounded domain variants of operators given on the whole space by their Fourier symbols. In order to capture non-periodic settings, as it is done classically, the corresponding version of $H$-convergence was used, called $H^s$-convergence here which was introduced in \cite{CCM} under a different name. The main results of the present research characterises $H^s$-convergence in terms of $H$-convergence and convergence in the weak-* topology. Moreover, these results are put into perspective of the notion of the Schur topology (two seemingly different Schur topologies were found that both characterise $H^s$-convergence). As a result, a homogenisation problem for the corresponding fractional heat equation has been solved. The final application provides a perspective to symmetric coefficients and characterises the suitable variant of $G$-convergence by mere energy convergence. 

The results are based on compactness results, Poincar\'e-type inequalities, and a suitable version of \v Zikov’s lemma recently used for the description of local $H$-convergence in terms of nonlocal $H$-convergence, i.e., convergence with respect to an appropriate Schur topology.

In comparison to earlier research, the improvement concerns the incorporation of the one-dimensional case, the possibility of introducing coefficients oscillating on the whole of $\Rd$, the complete description of $H^s$/$G^s$-convergence in terms of known, local topologies, the characterisation in terms of the Schur topology, and their application to a time-dependent partial differential equation.

  \paragraph{\bf Funding statement. }
    This work was supported by the Croatian Science Foundation under the projects IP-2022-10-5181 (HOMeoS) and IP-2022-10-7261 (ADESO).

\printbibliography 
\end{document}